\documentclass[12pt]{amsart}
\usepackage{amsmath,amsthm,amssymb,amsfonts}
\usepackage{latexsym}
\usepackage{enumerate}
\usepackage{verbatim}
\usepackage{multicol}
\usepackage{graphicx}
\usepackage{epsfig}
\usepackage{color}
\usepackage{bezier}
\usepackage{curves}
\usepackage{fullpage}
\usepackage{amssymb, amsfonts, amsmath, latexsym, verbatim, xspace, setspace, booktabs, tikz, xcolor, multicol}
\usetikzlibrary{calc}
\usepackage{float}

\newtheorem{theorem}{Theorem}

\newtheorem{remark}{Remark}
\newtheorem{lemma}[theorem]{Lemma}
\newtheorem*{lemma*}{Lemma}

\newtheorem{proposition}[theorem]{Proposition}
\newtheorem{corollary}[theorem]{Corollary}

\newtheorem{df}[theorem]{Definition}

\DeclareMathOperator{\comp}{{Comp}}

\DeclareMathOperator{\incg}{Inc}

\DeclareMathOperator{\sing}{Sing}

\newcommand\mydownarrow{\mathord{\downarrow }}

\newcommand{\QQ}{{\mathbb Q}}
\newcommand{\NN}{{\mathbb N}}
\newcommand{\ZZ}{{\mathbb Z}}

\begin{document}

\title{The structure of interval orders with no infinite antichain}

\author [M. Pouzet]{Maurice Pouzet}
\address{Univ. Lyon, Universit\'e Claude-Bernard Lyon1, CNRS UMR 5208, Institut Camille Jordan, 43, Bd. du 11 Novembre 1918, 69622
Villeurbanne, France et Department of Mathematics and Statistics, University of Calgary, Calgary, Alberta, Canada}

\author[I. Zaguia]{Imed Zaguia*}\thanks{*The author is supported by the Canadian Defence Academy Research Program and NSERC}
\address{Department of Mathematics \& Computer Science,  Royal Military College,
P.O.Box 17000, Station Forces, Kingston, Ontario, Canada K7K 7B4}
\email{zaguia@rmc.ca}

\date{\today}

\keywords{(partially) ordered set; chain; antichain; module, clan, modular decomposition; Gallai decomposition, interval order; scattered; Hausdorff rank}
\subjclass[2020]{06A6, 06F15}

\begin{abstract}
We prove that if \(G=(V,E)\) is a nonprime graph with either no infinite
independent set or no infinite clique, then every vertex of \(G\) belongs to a
maximal strong module distinct from \(V\). In particular, \(G\) admits a
Gallai decomposition.

As a consequence, we obtain that every interval order \(P\) with no infinite
antichain admits a Gallai decomposition. That is, \(P\) is a lexicographical
sum of interval orders distinct from \(P\) indexed by either a chain, an antichain, or a
prime interval order.

Next, we prove that every prime interval order with no infinite antichain is
at most countable and does not embed a copy of the chain of
rational numbers.

Finally, for each countable ordinal \(\alpha\), we construct a well-quasi-ordered
prime interval order \(P_\alpha\) whose chain of maximal antichains has
Hausdorff rank \(\alpha\).
\end{abstract}

\maketitle

\section{Introduction}

This paper is about the structure of interval orders of any cardinality. Let us recall that an ordered set $P$ is an \emph{interval order} if $P$ is isomorphic to a subset $\mathcal J$ of the set $Int(C)$ of nonempty intervals of a chain $C$, ordered as follows: if $I, J\in Int(C)$, then

\begin{equation}
I<J \mbox{  if  } x<y  \mbox{  for every  } x\in I  \mbox{  and every  } y\in J.
\end{equation}

See \cite{fi} for a wealth of  information about interval orders. 

In \cite{leblet-rampon}, the authors provided an inductive construction of finite interval orders. One of the main goals of this paper is to prove that interval orders with no infinite antichains can be constructed from prime interval orders with no infinite antichains using lexicographical sums.

Let $I$ be an ordered set such that $|I|\geq 2$ and let $\{P_{i}:=(V_i,\leq_i)\}_{i\in I}$ be a family of pairwise disjoint nonempty ordered sets
that are all disjoint from $I$. The \emph{lexicographical sum} $\displaystyle \sum_{i\in I} P_{i}$ is the ordered set defined on
$\displaystyle \cup_{i\in I} V_{i}$ by $x\leq y$ if and only if
\begin{enumerate}[(a)]
\item There exists $i\in I$ such that $x,y\in V_{i}$ and $x\leq_i y$ in $P_{i}$; or
\item There are distinct elements $i,j\in I$ such that $i<j$ in $I$,   $x\in V_{i}$ and $y\in V_{j}$.
\end{enumerate}

The ordered sets $P_{i}$ are called the \emph{components} of the lexicographical sum and the ordered set $I$ is the \emph{index set}.
If $I$ is a totally ordered set, then $\displaystyle \sum_{i\in I} P_{i}$ is called a \emph{linear sum}. At times we will also use the notation $\oplus_{i\in I}P_i$ to indicate a linear sum. On the other hand, if $I$ is an antichain, then $\displaystyle \sum_{i\in I} P_{i}$ is called a \emph{disjoint sum}.

We should mention that the class of interval orders is not closed under lexicographical sums. Indeed, the disjoint sum of two chains having at least two elements is not an interval order even though chains (and antichains) are interval orders. Still, we will see that, by restricting the type of lexicographical sums, we will be able to generate all interval orders with no infinite antichains from prime interval orders with no infinite antichains.

We recall the notions of modules and prime binary relations.

\begin{def}\label{def:module}
Let  $R:=(V,\rho)$ be a binary relation, i.e., $\rho\subseteq V\times V$. A  \emph{module} in $R$ is any subset $A$ of $V$ such that  \[(x\rho\, a \Leftrightarrow x\rho\, a')  \; \text{and} \; (a\rho\, x \Leftrightarrow a'\rho\, x) \; \text{for all} \; a,a'\in A \;\text{and}\; x\notin A.\]
\end{def}

Sometimes in the literature, modules  are called \emph{intervals},  \emph{autonomous} or \emph{partitive sets}. The empty set, the singletons in $V$ and the whole set $V$ are modules and are called \textit{trivial}. If $R$ has no nontrivial module, it is called \emph{prime} or \textit{indecomposable}.

\noindent For example, if $R:=(V,\leq)$ is a chain, its modules  are the ordinary intervals of the chain. Hence, a chain on at least three vertices is not prime.

The notion of module  goes back to Fra\"{\i}ss\'e \cite{fraisse53, fraisse84} and Gallai  \cite{gallai}. A fundamental decomposition result of a binary relation into modules was obtained by Gallai \cite{gallai} in the finite case (see \cite{kelly85}, \cite{HarjuRozenberg}, \cite{courcelle-delhomme} and \cite{ille-woodrow-2009, ille-woodrow-2011} for further extensions).

Let $P=(V,\leq)$ be an ordered set. An element $x\in V$ is \emph{singular} if the set of elements of $V$ incomparable with $x$, is an antichain. Equivalently, $x$ is singular if and only if $x$ belongs to a unique maximal antichain in $P$. We denote by $\sing(P)$ the set of singular vertices of $P$. Notice that in general $\sing(P)$ could be empty. But if we impose a certain condition on an interval order $P$, then $P$ has at least one singular element (Lemma \ref{lem:minimal-element}). We will prove (see Lemma \ref{lem:singular}) that if $P$ is a prime interval order, then $(\sing(P),\leq)$ is a chain.

Let $IO^{<\omega}$ be the class of interval orders with no infinite antichains. Our first result is this

\begin{theorem}\label{thm:main-lex}Let $P$ be an ordered set. Then $P\in IO^{<\omega}$ if and only if $P$ is the lexicographical sum $\displaystyle \sum_{i\in Q} P_{i}$ where each $P_i\in IO^{<\omega}$ and $Q$ is either
\begin{enumerate}[$(1)$]
  \item a chain,
  \item a finite antichain, in which case all the $P_i$'s but at most one are finite antichains,
  \item a prime interval order belonging to \(IO^{<\omega}\) and each $P_i$ for $i\not \in \sing(Q)$ is a finite antichain.
\end{enumerate}
\end{theorem}

A consequence is this.

\begin{corollary}
Every member of $IO^{<\omega}$ can be obtained from chains, finite antichains and prime members of $IO^{<\omega}$ by iterating lexicographical sums satisfying the restrictions imposed in the theorem.
\end{corollary}

The proof is presented in subsection \ref{subsection:interval-gallai}. Theorem \ref {thm:main-lex} is a consequence of the following result.

\begin{theorem}\label{cor:max-strong}If $G=(V,E)$ is a nonprime graph with either no infinite independent set or no infinite clique, then every vertex of $G$ is contained in a maximal strong module distinct from $V$. In particular, $G$ has a Gallai decomposition.
\end{theorem}

The proof is presented in section \ref{section:tree}.

Let $\mathcal{G}$ be the collection of graphs as follows. Every prime graph with no infinite independent set, every clique, and every finite independent set belong to  $\mathcal{G}$. The lexicographical sum of members of $\mathcal{G}$ indexed by any member of $\mathcal{G}$ belongs to $\mathcal{G}$. No other graphs belong to $\mathcal{G}$. 

\begin{corollary}The class $\mathcal{G}$ is the class of graphs with no infinite independent sets.
\end{corollary}

Let $P=(V,\leq)$ be an ordered set. Denote by $AM(P)_{\leq}$ the set $AM(P)$ of maximal antichains of $P$ ordered as follows. If $A,B\in AM(P)$:
\begin{equation}\label{eqordreinter}
A\leq_{AM} B \; \mbox{ if for every }\;  a\in A\;  \mbox{ there exists }\; b\in B \; \mbox{such that }\; a\leq b.
\end{equation}

It is easy to see that $\leq_{AM}$ is an order on $AM(P)$. Furthermore, if $P$ is an interval order, then $\leq _{AM}$ is a linear order.\\

Recall that an ordered set is \emph{scattered} if it does not embed the chain of rational numbers \(\QQ\). A characterization of scattered linear orders was obtained by Hausdorff \cite{hausdorff} in 1908 (see Theorem \ref{thm:hausdorff-scattered} of this paper), and rediscovered by Erd\H{o}s
and Hajnal \cite{eh63} in 1962.

In \cite{uri-bonnet}, the authors characterized the class of scattered ordered sets with no infinite antichain as the smallest class containing the
well-founded ordered sets with no infinite antichain and closed under the following operations:
\begin{enumerate}[(1)]
\item taking the dual;
\item lexicographical sums;
\item extensions.
\end{enumerate}

Here, an ordered set is \emph{well-founded} if it contains no chain order-isomorphic to the chain of negative integers, and
\(P=(V,\preceq)\) is an \emph{extension} of \(P=(V,\leq)\) if
\[
x\leq y \Longrightarrow x\preceq y.
\]

Let \(P=(V,\leq)\) be an ordered set. For \(a\in V\), let
\[
D(a):=\{x\in V:x<a\}
\qquad\text{and}\qquad
U(a):=\{x\in V:x>a\}.
\]
Scattered interval orders with no infinite antichains admit the following characterization.

\begin{theorem}\label{thm:scattered}Let $P$ be an interval order with no infinite antichains. The following propositions are equivalent.
\begin{enumerate}[$(i)$]
  \item $P$ is scattered.
  \item $AM(P)_{\leq}$ is a scattered chain.
  \item $(\{U(x) : x\in V\},\subseteq)$ is a scattered chain.
   \item $(\{D(x) : x\in V\},\subseteq)$ is a scattered chain.
  \item $P$ is isomorphic to a collection of intervals of some scattered chain.
\end{enumerate}
\end{theorem}

Notice that in Theorem \ref{thm:scattered}, the implication
\[
P \text{ non-scattered } \Longrightarrow AM(P)_\le \text{ non-scattered}
\]
requires the assumption that \(P\) has no infinite antichain. Indeed, define an order relation $\leq$ on the set $\{0,1\}\times \QQ$ as follows: $(i,r)<(i',r')$ if and only if $i=0$ and $i'=1$ and $r<r'$ in $\QQ$. The ordered set $B(\QQ):=(\{0,1\}\times \QQ,\leq)$, called the \emph{incidence bipartite ordered set} of the chain $\QQ$, is a scattered interval order but $AM_{\leq}(B(\QQ))$ is not.\\

The proof of Theorem \ref{thm:scattered} is presented in section \ref{section:scattered}.

The previous theorem strongly restricts the structure of scattered interval orders. For prime interval orders, we obtain an even stronger conclusion.

\begin{theorem}\label{thm:prime-scattered}A prime interval order with no infinite antichains is at most countable and scattered.
\end{theorem}

The conclusion of Theorem~\ref{thm:prime-scattered} becomes false if the
assumption of having no infinite antichain is removed. Indeed, let
\(Q:=\QQ\oplus 1,\) that is, the chain of rational numbers with a largest element \(m\) added.

For each \(r\in \QQ\), add a new element \(q_r\) such that \(r<q_r,\) and \(q_r\) is incomparable with every element \(s\in Q\) satisfying \(r\leq s.\)
Furthermore, the elements \(q_r\) are pairwise incomparable. Next, add a new element \(q_m\) such that \(q_m<m,\) and \(q_m\) is incomparable with every element of \(\QQ\) and with every element \(q_r\) with \(r\in \QQ\). Finally, take the transitive closure of the resulting relation.

The resulting ordered set is a prime interval order, is not scattered, and contains an infinite antichain.

\medskip

Another important contribution of Hausdorff to the study of scattered linear
orders is the introduction of the Hausdorff rank (see \cite{fraissetr},
Chapter~6). The scattered chains of rank \(0\) are precisely the finite
chains, the chain of natural numbers \(\NN\), its dual \(\NN^*\), and the
chain of integers \(\ZZ\). Moreover, every scattered chain of nonzero rank can
be decomposed into a sum of scattered chains of strictly smaller rank (see
\cite{fraissetr}, Section~6.2.6, p.~171).

An ordered set \(P\) is \emph{well-quasi-ordered} if every nonempty subset of
\(P\) has finitely many minimal elements (this number being nonzero). Equivalently, \(P\) is
well-quasi-ordered if and only if it contains no chain order-isomorphic to the
chain of negative integers and no infinite antichain.

Our final result shows that prime interval orders may still exhibit arbitrary
countable scattered complexity.

\begin{theorem}\label{thm:main-rank}
For every countable ordinal \(\alpha\), there exists a well-quasi-ordered
prime interval order \(P_\alpha\) whose chain of maximal antichains has
Hausdorff rank \(\alpha\).
\end{theorem}

The paper is organized as follows. In section \ref{section:prereq} we introduce some basic definitions and facts. In section \ref{section:scattered} we provide a proof of Theorem \ref{thm:scattered}. In section \ref{section:prime-scattered} we prove that prime interval orders with no infinite antichains are scattered. Section \ref{section:countable-scattered} is devoted to proving that prime interval orders with no infinite antichains are at most countable. Section \ref{section:example} of the paper is devoted to describing examples of well-quasi-ordered prime interval orders whose chain of maximal antichains has arbitrary countable Hausdorff rank, thus providing a proof of Theorem \ref{thm:main-rank}.

\section{Prerequisites}\label{section:prereq}
We introduce some basic definitions and facts that we will need for the proofs.

\subsection{Well-founded ordered sets and decomposition into levels}\label{sub:well-founde}
We recall that an ordered set is \emph{well-founded} if it has no chain order-isomorphic to the chain of negative integers. The decomposition of a well-founded ordered set $P=(V,\leq)$ into \textit{levels} is the sequence $P_{0},\cdots,P_{l},\cdots$
defined by transfinite induction by the formula
\begin{equation*}
P_{l}:=\min(P\setminus \bigcup \{P_{l^{\prime }}:l^{\prime }<l\}).
\end{equation*}
In particular, $P_{0}=\min(P)$, is the set of minimal elements of $P$. The \textit{height} of $P$, denoted by $h(P)$, is the smallest ordinal $l$ such that $P_{l}= \varnothing $. Hence, $V=\cup \{P_{l}:l\leq h(P)\}$. A well-founded ordered set is \emph{level finite} if all its levels are finite.

\subsection{Properties of modules}\label{sub:module}
We now recall some properties of modules in an ordered set. The proof of the following lemma is easy and is left to the reader (see Lemma 4.1 \cite{courcelle-delhomme}).

\begin{lemma}\label{lem:1}Let $P=(V,\leq)$ be an ordered set. The following propositions are true.
\begin{enumerate}[$(1)$]
\item If $M$ is a module, then $M$ is convex, that is for all $u,v\in M$ if $u<v$, then $\{x\in V:u\leq x\leq v\}\subseteq M$.
  \item The intersection of a nonempty set of modules is a module (possibly empty).
  \item The union of two modules with nonempty intersection is a module.
  \item For two modules $M$ and $N$, if $M\setminus N\neq \varnothing$, then $N\setminus M$ is a module.
\end{enumerate}
\end{lemma}

\begin{lemma}\label{lem:module-antichain} Let $P$ be a poset. Then:
\begin{enumerate}[$(1)$]
\item The union $A\cup B$ of two antichains and modules $A$ and $B$ of $P$ such that  $A\cap B\neq \emptyset$   is an antichain and a module of $P$.
\item Every antichain and a module of $P$ is contained in an antichain and module of $P$, maximal with respect to set inclusion.
\end{enumerate}
\end{lemma}
\begin{proof} $(1)$ Since \(A\) and \(B\) are modules and \(A\cap B\neq\emptyset\), it follows
from Lemma~\ref{lem:1}(3) that \(A\cup B\) is a module.

We prove that \(A\cup B\) is an antichain. Let \(a,b\in A\cup B\). If
\(a,b\in A\), then \(a\) and \(b\) are incomparable because \(A\) is an
antichain. Similarly, if \(a,b\in B\), then \(a\) and \(b\) are incomparable.

It remains to consider the case \(a\in A\setminus B\) and \(b\in B\setminus A\).

Choose \(c\in A\cap B\). Since \(a,c\in A\) and \(A\) is an antichain, \(a\)
is incomparable with \(c\). Since \(B\) is a module and \(a\notin B\), the
relation of \(a\) to all elements of \(B\) is the same. Hence, because \(a\)
is incomparable with \(c\in B\), it follows that \(a\) is incomparable with
every element of \(B\), in particular with \(b\). Therefore \(A\cup B\) is an
antichain.\\
$(2)$ This second assumption of the lemma follows readily from the previous one and Zorn's Lemma.
\end{proof}

Let $P=(V,\leq)$ be a poset and $x, y \in P$. We set $x \equiv_P y$ if there exists an antichain and a module $A$ of $P$ containing $x$ and $y$. Clearly, $\equiv_P$ is reflexive and symmetric. It follows from Lemma \ref{lem:module-antichain} that $\equiv_P$ is also transitive. Hence, $\equiv_P$ is an equivalence relation on $V$. The equivalence class of each element $v\in V$ is an antichain and a module, maximal with respect to set inclusion,  containing $v$. We denote by $P/\equiv_P$ the quotient ordered set and let $\psi:P \rightarrow P/\equiv_P$ be the \emph{quotient map} and notice that $\psi$ is order preserving and surjective.

\begin{lemma}\label{lem:gallaiquotient}Let $P$ be a poset. Then:
\begin{enumerate}[$(1)$]
\item $P$ is the lexicographical sum $\sum_{i\in Q} A_i$ of antichains $A_i$ indexed by a poset $Q$ in which no nontrivial module is an antichain.
\item If $P$ is an interval order, then $Q$ is also an interval order.
\end{enumerate}
\end{lemma}

\begin{proof}
$(1)$  $P$ is the lexicographical sum of its equivalence classes indexed by $P/\equiv_P$. If \(A\) is an antichain and a module of \(P/\equiv_P\), then
\(\psi^{-1}(A)\) is an antichain and a module of \(P\). If \(|A|\ge 2\), then
\(\psi^{-1}(A)\) strictly contains each equivalence class lying over an element
of \(A\), contradicting the maximality of the equivalence classes. Hence, \(|A|\le 1\). So we may set   $Q:=P/\equiv_P$.

For $(2)$, suppose that \(P\) is an interval order. Choose one representative
\(x_i\in A_i\) from each equivalence class \(A_i\). Let \(R=\{x_i:i\in Q\}\).
We claim that \(Q\) is isomorphic to the induced subposet \(P[R]\). Indeed,
because the equivalence classes are modules, if \(A_i\neq A_j\), then the
order relation between \(A_i\) and \(A_j\) is uniform: either every element of
\(A_i\) is below every element of \(A_j\), or every element of \(A_j\) is below
every element of \(A_i\), or all elements of \(A_i\) are incomparable with all
elements of \(A_j\). Hence, \[
A_i\le_Q A_j
\quad\Longleftrightarrow\quad
x_i\le_P x_j.
\]
Thus, \(Q\cong P[R]\).

Since interval orders are hereditary under taking induced subposets, \(P[R]\)
is an interval order. Therefore \(Q\) is an interval order.
\end{proof}

\begin{lemma}\label{lem:module-linearsum}
Let $I$ be linear order and $(P_i=(V_i,\leq_i))_{i\in I}$ be a sequence of prime ordered sets. Suppose that for all $i\in I$, $|V_i|\geq 4$. If $M$ is a nontrivial module in the linear sum $\sum_{i\in I} P_{i}$, then there exists an interval $J\neq I$ such that  $M=\cup_{j\in J} V_j$.
\end{lemma}

\begin{proof}
Let \(P:=\sum_{i\in I}P_i\). Let \(M\) be a nontrivial module of \(P\). We claim first that, for every
\(i\in I\),
\(
M\cap V_i\neq\emptyset
\quad\Longrightarrow\quad
V_i\subseteq M.
\)

Fix \(i\in I\) such that \(M\cap V_i\neq\emptyset\). Since \(M\) is a module
of \(P\), the set \(M\cap V_i\) is a module of \(P_i\). Because \(P_i\) is
prime, we have \(|M\cap V_i|\in\{1,|V_i|\}\).

Suppose, for a contradiction, that \(M\cap V_i=\{x\}\). Since \(M\) is nontrivial, there exists \(z\in M\setminus\{x\}\). Then
\(z\in V_j\) for some \(j\neq i\).

Assume first that \(i<j\). Then \(x<z\).

Let \(y\in V_i\setminus\{x\}\). Since \(y<z\) and \(y\notin M\), while
\(x,z\in M\), the module property of \(M\) forces \(y<x\). Thus, every element of \(V_i\setminus\{x\}\) is below \(x\). Hence, \(V_i\setminus\{x\}\) is a module of \(P_i\) and is nontrivial because \(|V_i|\ge 4\), contradicting that \(P_i\) is prime and .

The case \(j<i\) is symmetric. Then \(z<x\), and the same argument shows that every element of \(V_i\setminus\{x\}\) is above \(x\). Again
\(V_i\setminus\{x\}\) is a nontrivial module of \(P_i\), a contradiction.

Therefore the singleton case cannot occur, and so \(M\cap V_i=V_i\). This proves the claim.

Now let \(J:=\{i\in I:M\cap V_i\neq\emptyset\}\). By the claim, \(M=\cup_{j\in J}V_j\).

It remains to show that \(J\) is an interval of \(I\). Let \(i,k\in J\) and suppose \(i<j<k\).

Choose \(a\in V_i\) and \(b\in V_k\). Then \(a,b\in M\). If \(j\notin J\),
choose \(y\in V_j\). Then \(y\notin M\), and in the linear sum we have \(a<y<b\).  This contradicts the fact that \(M\) is a module, since an element outside \(M\) must have the same relation to every element of \(M\). Therefore \(j\in J\), and \(J\) is an interval.

Finally, since \(M\) is nontrivial, \(M\neq V(P)\). Hence, \(J\neq I\). Thus, there exists an interval \(J\neq I\) such that \(M=\cup_{j\in J}V_j\).
\end{proof}

\subsection{Basic facts about interval orders}

The following are well-known characterisations of interval orders. For a proof see \cite{fi}.

\begin{lemma}\label{lem:intervalorder-prop}
  Let $P$ be an ordered set. The following propositions are equivalent.
  \begin{enumerate}[(i)]
    \item $P$ is an interval order.
    \item $P$ does not embed ${\rm \bf 2} \oplus {\rm \bf 2}$, the disjoint sum of two $2$-element chains.
    \item $AM(P)_{\leq}$ is a chain.
    \item $(\{U(x) : x\in V\},\subseteq)$ is a chain.
    \item $(\{D(x) : x\in V\},\subseteq)$ is a chain.
  \end{enumerate}
\end{lemma}

We will need the following lemma.

\begin{lemma}\label{lem:minimal-element}Let $P=(V,\leq)$ be a well-founded level finite interval order. Let $x\in \min(P)$ with the property that $U(x)$ is a maximal element in the ordered set \break $(\{U(t): t\in \min(P)\},\subseteq)$. Then $U(x)=V\setminus \min(P)$. Equivalently, $\min(P)$ is the only maximal antichain of $P$ containing $x$.
\end{lemma}

\begin{proof}
Let \(v\in V\setminus \min(P)\). We prove that \(x<v\).

Suppose, for a contradiction, that \(v\) is incomparable with \(x\). Since
\(P\) is well-founded and \(v\notin \min(P)\), there exists
\(y\in \min(P)\) such that \(y<v\). Since \(v\) is incomparable with \(x\), we have \(y\neq x\). Now \(v\in U(y)\), while \(v\notin U(x)\). Hence, \(U(x)\neq U(y)\). By Lemma~\ref{lem:intervalorder-prop}(iv), the sets \(U(x)\) and \(U(y)\) are
comparable under set inclusion. Since \(v\in U(y)\setminus U(x)\), we must have
\(
U(x)\subsetneq U(y).
\)
This contradicts the choice of \(x\), because \(U(x)\) was assumed maximal
among the sets
\(
\{U(t):t\in \min(P)\}.
\)
Therefore \(v\) cannot be incomparable with \(x\). Since \(x\) is minimal and
\(v\notin \min(P)\), it follows that \(x<v\). Thus, \(V\setminus \min(P)\subseteq U(x)\).
The reverse inclusion is immediate, since no minimal element lies above \(x\).
Hence, \(U(x)=V\setminus \min(P)\).

Finally, this is equivalent to saying that \(\min(P)\) is the only maximal
antichain containing \(x\). Indeed, if \(U(x)=V\setminus\min(P)\), then \(x\)
is comparable with every nonminimal element, so any antichain containing \(x\)
is contained in \(\min(P)\). Hence, the unique maximal antichain containing
\(x\) is \(\min(P)\).

Conversely, if some \(v\in V\setminus\min(P)\) were incomparable with \(x\),
then \(\{x,v\}\) could be extended to a maximal antichain containing \(x\),
different from \(\min(P)\), a contradiction.
\end{proof}

\section{A proof of Theorem \ref{thm:scattered}}\label{section:scattered}

Before proceeding to the proof of Theorem~\ref{thm:scattered}, we need a
preliminary result. Let \(P=(V,\leq)\) be an ordered set. An \emph{initial
segment} of \(P\) is a subset \(I\subseteq V\) such that, whenever \(x\in I\)
and \(y\leq x\), we have \(y\in I\). We denote by \(I(P)\) the set of initial
segments of \(P\), and by $(I(P),\subseteq)$ the set \(I(P)\) ordered by set inclusion.

For \(A\subseteq V\), set
\[
\mydownarrow A:=\{x\in V:x\leq a \text{ for some }a\in A\}.
\]
Clearly, \(\mydownarrow A\) is an initial segment of \(P\).

If \(A\) and \(B\) are maximal antichains of \(P\) and \(A\leq_{AM}B\), then \(\mydownarrow A\subseteq \mydownarrow B\). Moreover, if \(A\neq B\), then
\(
\mydownarrow A\subsetneq \mydownarrow B.
\)
Therefore the map
\[
A\longmapsto \mydownarrow A
\]
is an order-embedding of \(AM(P)_\leq\) into \((I(P),\subseteq)\). Hence, if
\((I(P),\subseteq)\) is scattered, then \(AM(P)_\leq\) is scattered.

The following result is part of the second theorem on page~149 of
\cite{bonnet-pouzet}; see also Theorem~2 of \cite{pouzet-zaguia}.

\begin{theorem}\label{thm:bonnet-pouzet}Let $P=(V,\leq)$ be an ordered set. The following are equivalent.
  \begin{enumerate}[$(a)$]
    \item $P$ is scattered and with no infinite antichains.
    \item $(I(P),\subseteq)$ is scattered.
  \end{enumerate}
\end{theorem}

We will also need the following two lemmas. Let \(C\) be a chain and let
\(\mathcal J\) be a collection of nonempty intervals of \(C\). We say that
\((C,\mathcal J)\) is \emph{discriminating} if the following properties hold:

\begin{enumerate}[(1)]
\item Every element of \(C\) belongs to some interval in \(\mathcal J\);

\item Every two distinct elements \(x,y\in C\) are \emph{separated}, that is,
there exist intervals \(I_x,I_y\in\mathcal J\) such that
\(
x\in I_x,\quad y\in I_y,
\quad\text{and}\quad
I_x\cap I_y=\varnothing;
\)

\item \(\mathcal J\) has the \emph{finite intersection property}, that is,
for every subfamily \(\mathcal J'\subseteq\mathcal J\), the intersection \(\cap\mathcal J'\) is nonempty whenever every finite subfamily of \(\mathcal J'\) has nonempty intersection.
\end{enumerate}

The following lemma (Lemma~3 of \cite{zaguia2008}) gives a necessary and
sufficient condition for an interval order to admit a unique representation by
its chain of maximal antichains.

\begin{lemma}[\cite{zaguia2008}]\label{lem:standard}
Let \(P\) be an interval order, \(C:=AM(P)_\leq\),  and, for every \(x\in P\), let
\(C(x):=\{X\in C:x\in X\}\). Set \(\mathcal J:=\{C(x):x\in P\}\).

Then \(\mathcal J\) is a discriminating family of nonempty intervals of \(C\),
and the map \(\varphi:P\to\mathcal J\) defined by \(\varphi(x):=C(x)\) is an order-preserving surjection. Moreover, \(\varphi\) is an order
isomorphism if and only if \(P\) has no \(2\)-element antichain as a module.
\end{lemma}

The next lemma extends the previous one to interval orders having a
\(2\)-element antichain as a module.

\begin{lemma}\label{lem:extended-standard}
Let \(P=(V,\leq)\) be a scattered interval order with no infinite antichain.
Then \(P\) is isomorphic to a family of intervals of a scattered chain.
\end{lemma}

\begin{proof}
Let \(Q:=P/{\equiv_P}\), where the equivalence classes are the maximal
modules of \(P\) which are antichains. By Lemma~\ref{lem:gallaiquotient},
\(
P=\sum_{q\in Q} A_q,
\)
where \(A_q\) is the equivalence class over \(q\) is a finite antichain, and \(Q\) has no nontrivial
module which is an antichain. Write
\(
q=\{x_1^q,\dots,x_{n(q)}^q\}.
\)

Since \(P\) is scattered and has no infinite antichain, the same is true of
\(Q\). Indeed, any chain or antichain in \(Q\) lifts to one in \(P\) by
choosing representatives. Hence, by Theorem~\ref{thm:bonnet-pouzet},
\((I(Q),\subseteq)\) is scattered. Since
\(
A\longmapsto \mydownarrow A
\)
embeds \(AM(Q)_\leq\) into \((I(Q),\subseteq)\), the chain
\(
C:=AM(Q)_\leq
\)
is scattered.

For each \(q\in Q\), put
\(
I_q:=\{A\in C:q\in A\}.
\)
By Lemma~\ref{lem:standard}, since \(Q\) has no nontrivial antichain module,
the map
\(
q\longmapsto I_q
\)
is an isomorphism from \(Q\) onto a family of nonempty intervals of \(C\).

For each \(q\in Q\), define the following initial segments of \(C\).
\[
L_q:=\{A\in C:A<B\text{ for every }B\in I_q\},
\qquad
R_q:=L_q\cup I_q.
\]
Then
\(
I_q=R_q\setminus L_q
\quad\text{and}\quad
L_q\subsetneq R_q.
\)

Let \(\mathcal I(C)\) be the chain of initial segments of \(C\), ordered by
set inclusion. Since \(C\) is scattered, \(\mathcal I(C)\) is scattered (for
a proof see section 6.1.4 of \cite{fraissetr}).  For each
\(H\in\mathcal I(C)\), set
\[
Q_H:=\{q\in Q:L_q=H\}.
\]
Suppose \(r<s\) in \(Q_H\). Then \(I_r<I_s\), hence \(R_r\subseteq L_s\).
But \(L_s=L_r\), so \(R_r\subseteq L_r\), contradicting \(L_r\subsetneq R_r\).
Therefore no two distinct elements of \(Q_H\) are comparable. Hence, \(Q_H\) is an
antichain. Since \(Q\) has no infinite antichain, \(Q_H\) is finite.

For each \(H\in\mathcal I(C)\), let
\[
X_H:=\{p_i^q:q\in Q_H,\ 1\leq i\leq n(q)\}
\]
be a finite set of new points, one for each \(x_i^q\in q\), all chosen
pairwise distinct. Choose an arbitrary linear order on \(X_H\). Let \(0_H\) be a new point,
chosen distinct from all the points \(p_i^q\), and set
\[
F_H:=\{0_H\}\oplus X_H,
\]
the linear sum in which \(0_H\) lies below every point of \(X_H\). If
\(Q_H=\emptyset\), then \(X_H=\emptyset\), and hence \(F_H=\{0_H\}\).

Now define
\[
D:=\sum_{H\in\mathcal I(C)}F_H.
\]
Since \(\mathcal I(C)\) is scattered and every \(F_H\) is finite, \(D\) is
scattered.

For \(x=x_i^q\in q\), define
\[
J_x:=[p_i^q,0_{R_q}]_D.
\]
This is a well-defined interval of \(D\), because
\[
p_i^q\in F_{L_q},
\qquad
0_{R_q}\in F_{R_q},
\qquad
L_q\subsetneq R_q,
\]
and hence \(p_i^q<0_{R_q}\) in \(D\).

We claim that
\[
x<y
\quad\Longleftrightarrow\quad
J_x<J_y.
\]
Let \(x=x_i^q\) and \(y=x_j^r\). If \(x<y\) in \(P\), then \(q<r\) in \(Q\), hence
\(I_q<I_r\). Therefore
\(
R_q\subseteq L_r,
\)
so
\(
0_{R_q}<p_j^r,
\)
and consequently \(J_x<J_y\).

Conversely, if \(J_x<J_y\), then
\(
0_{R_q}<p_j^r.
\)
Since these points lie respectively in \(F_{R_q}\) and \(F_{L_r}\), we get
\(
R_q\subseteq L_r.
\)
Thus, \(
I_q<I_r,
\)
so \(q<r\), and therefore \(x<y\) in the lexicographical sum \(P\).

It remains to prove that the map \(x\mapsto J_x\) is injective. Suppose
\(
J_x=J_y.
\)
Write
\[
x=x_i^q,\qquad y=x_j^r.
\]
Then
\[
J_x=[p_i^q,0_{R_q}]_D,
\qquad
J_y=[p_j^r,0_{R_r}]_D.
\]

Since two closed intervals in a chain are equal only if they have the same
left endpoint and the same right endpoint, we get
\(
p_i^q=p_j^r
\quad\text{and}\quad
0_{R_q}=0_{R_r}.
\)

The points \(p_i^q\) were chosen pairwise distinct, so
\(
p_i^q=p_j^r
\)
implies
\(
q=r
\quad\text{and}\quad
i=j.
\)
Therefore
\(
x=x_i^q=x_j^r=y.
\)
Thus, the map \(x\mapsto J_x\) is injective.

Thus, \(P\) is isomorphic to the interval order induced by the family
\(
\{J_x:x\in P\}
\)
of intervals of the scattered chain \(D\).
\end{proof}

We now present a proof of Theorem \ref{thm:scattered}.

\begin{proof}
We first prove
\[
(i)\Longleftrightarrow(ii),\qquad
(i)\Longrightarrow(iii),\qquad
(i)\Longrightarrow(iv),
\]
and then
\[
(iii)\Longrightarrow(i),\qquad
(iv)\Longrightarrow(i),
\]
and finally \((i)\Longleftrightarrow(v)\).

\((i)\Rightarrow(ii)\). Suppose that \(P\) is scattered. Since \(P\) has no infinite antichain,
Theorem~\ref{thm:bonnet-pouzet} implies that \((I(P),\subseteq)\) is scattered. The map
\[
A\longmapsto \mydownarrow A
\]
is an order-embedding of \(AM(P)_\leq\) into \((I(P),\subseteq)\). Hence, \(AM(P)_\leq\) is scattered.

\((ii)\Rightarrow(i)\). Suppose that \(AM(P)_\leq\) is scattered. We show that \(P\) is
scattered. Let \(L\) be a chain of \(P\). For each \(x\in L\), choose a
maximal antichain \(A_x\) containing \(x\). We claim that the map \(x\longmapsto A_x\) is an order-embedding of \(L\) into \(AM(P)_\leq\).

Indeed, if \(x<y\) in \(L\), then \(A_x\neq A_y\). Since \(P\) is an interval order, \(AM(P)_\leq\) is a chain. Thus, either
\(
A_x<_{AM}A_y
\quad\text{or}\quad
A_y<_{AM}A_x.
\)
The second inequality is impossible. For if \(A_y\leq_{AM}A_x\), then, since
\(y\in A_y\), there exists \(a\in A_x\) such that \(y\leq a\). But \(x\in A_x\) and \(x<y\leq a\), contradicting the fact that \(A_x\) is an
antichain. Hence, \(
A_x<_{AM}A_y.
\)
Therefore, \(L\) embeds into \(AM(P)_\leq\). Since \(AM(P)_\leq\) is scattered,
\(L\) is scattered. Thus, every chain of \(P\) is scattered, so \(P\) is
scattered.

\((i)\Rightarrow(iii)\) Suppose that \(P\) is scattered. Then the dual \(P^*\) is also scattered
and has no infinite antichain. By Theorem~\ref{thm:bonnet-pouzet}, \((I(P^*),\subseteq)\) is scattered. For every \(x\in P\),
\[
U_P(x)=\{y\in P:x<y\}
\]
is an initial segment of \(P^*\). Hence, \((\{U(x):x\in V\},\subseteq)\) is a subchain of \((I(P^*),\subseteq)\), and is therefore scattered.

\((i)\Rightarrow(iv)\). Similarly, for every \(x\in P\),
\(
D_P(x)=\{y\in P:y<x\}
\)
is an initial segment of \(P\). Since \((I(P),\subseteq)\) is scattered, the
subchain \((\{D(x):x\in V\},\subseteq)\) is scattered.

\((iii)\Rightarrow(i)\). Suppose that \((\{U(x):x\in V\},\subseteq)\) is scattered. Let \(L\) be a chain of \(P\). If \(x<y\) in \(L\), then \(U(y)\subsetneq U(x)\), because \(U(y)\subseteq U(x)\), while \(y\in U(x)\) but \(y\notin U(y)\). Thus, \[
x\longmapsto U(x)
\]
is an order-reversing embedding of \(L\) into the scattered chain \((\{U(x):x\in V\},\subseteq)\). Hence, \(L\) is scattered. Therefore \(P\) is scattered.

\((iv)\Rightarrow(i)\). Suppose that \((\{D(x):x\in V\},\subseteq)\) is scattered. Let \(L\) be a chain of \(P\). If \(x<y\) in \(L\), then
\(D(x)\subsetneq D(y)\), because \(D(x)\subseteq D(y)\), while \(x\in D(y)\) but \(x\notin D(x)\). Thus, \[
x\longmapsto D(x)
\]
is an order-embedding of \(L\) into the scattered chain \((\{D(x):x\in V\},\subseteq)\).
Hence, \(L\) is scattered. Therefore \(P\) is scattered.

The four conditions \((i), (ii), (iii), (iv)\) are therefore equivalent.\\

We prove $(v) \Rightarrow (i)$. Suppose $P$ is isomorphic to a collection of intervals $\mathcal{J}$ of some scattered chain $C$ and let $\varphi$ be such an isomorphism. Let $L$ be a chain of $P$. For each $l\in L$ let $I_l$ be the interval of $C$ such that $\varphi(l)=I_l$. It follows from our assumption that $L$ is a chain that the intervals $I_l$, for $l\in L$ are pairwise disjoint. Choosing an arbitrary element $x_l\in I_l$ for every $l\in L$ we have  that $L$ is isomorphic to $\displaystyle \sum_{l\in L} \{x_l\}$. Hence, $L$ is isomorphic to a subchain of the scattered chain $C$ and is therefore scattered.\\
To prove $(i) \Rightarrow (v)$ Suppose $P$ is scattered. It follows from $(i)\Rightarrow (ii)$ that $AM(P)_{\leq}$ is scattered. Then use Lemma \ref{lem:extended-standard}.
\end{proof}

\section{Prime interval orders with no infinite antichains are scattered}\label{section:prime-scattered}

We provide a proof of the statement in the title of this section. Our original proof was based on a necessary and sufficient condition for an interval order to have a unique representation by intervals from its chain of maximal antichains (see  Lemma \ref{lem:standard}). The following simpler proof was suggested to us by one of the referees.

\begin{proof}
Let $P=(V,\leq)$ be a prime interval order with no infinite antichains. Let $K$ be a chain in $P$ that is order-isomorphic to the chain of rational numbers $\QQ$.

Suppose, by way of contradiction, that every element $c\in V$ that is comparable to some, but not all, elements of $K$ is incomparable with exactly one element $c_K \in K$. Let $H$ be the convex hull of
\[K \cup \{c \in V : \exists k\in K, c_K \in K \mbox{ and } c \mbox{ is comparable to } k \mbox{ and } c \mbox{ is incomparable with } c_K\}.\]
For $c \in K$, set $c_K := c$. Let $z \in V \setminus H$ be comparable to some element $c\in H$, without loss of generality, $z>c$. Because $P$ is an interval order the down sets are totally ordered by set inclusion and hence $c$ is larger than all elements of $\{k \in K : k<c_K \}\neq \varnothing$. Because $z\not \in H$, $z$ is comparable to all elements of $K$. Indeed, any element incomparable with some but not all elements of \(K\)
belongs to the defining set of \(H\). Because $H$ is convex, we obtain $z$ is larger than all elements of $K$. By definition, for all $h\in H$, there exists a $c'$ that is comparable to all elements in $K \setminus \{c'_K\}$ such that $h \leq c'$. Since  $c'$ is below all elements of $\{k \in K : k>c_K \}$ and $z$ is larger than all elements of the latter set we infer that $h<z$. Consequently, $H$ is a module in $P$, and, because $P$ is prime, we have $H = V$. Hence, every element of $V \setminus K$ is in $H$.

Moreover, if there was no $c\in V$ that is not comparable to all elements of $K$, then $P$ would be a lexicographic sum over an infinite chain and hence not prime. Hence, there exists $c\in V$ that is not comparable to all elements of $K$. Let $L_c$ be the set of all elements $h\in H$ such that $\{k \in K : k>c_K\} = \{k\in K : k>h\}$ and $\{k \in K : k<c_K \} = \{k \in K : k<h\}$. Then $\{c_K ,c\} \subseteq L_c$ and $L_c\neq V$. Let $z$ be strictly above an element in $L_c$. If $z$ is not comparable to all elements of $K$, then there is a $k \in K$ such that $z, z_K >k>c_K$ and $z$ is above all elements of $L_c$. If $z$ is comparable to all elements of $K$, then there is a $k \in K$ such that $z>k>c_K$ and $z$ is above all elements of $L_c$. Hence, $L_c$ is a nontrivial module in $P$, a contradiction.

Thus, for every chain $K$ in $P$ that is order isomorphic to $\QQ$, there is an $x^K$ that is not comparable to an interval of positive length in $K$. This interval contains a subset $K_1$ that is order-isomorphic to $\QQ$. Once $K_n$ a chain in $P$ that is order-isomorphic to $\QQ$ has been chosen, let $x_n$ be an element $x^{K_n}$ that is not comparable to an interval $J_n$ of positive length in $K_n$, and let $K_{n+1}$ be a subset of $J_n$ that is isomorphic to $\QQ$. Now $\{x_n : n \in \NN\}$ cannot be an antichain, so there are $n<m$ such that $x_n$ is comparable to $x_m$. By definition neither of $x_n$ and $x_m$ is comparable to any element of the infinite chain $K_m$. Hence, $P$ has a disjoint sum of two 2-element chains, contradicting that $P$ is an interval order. Therefore $P$ cannot contain copies of the chain $\QQ$ of rational numbers.
\end{proof}

\section{Prime interval orders with no infinite antichains are at most countable}\label{section:countable-scattered}

Let $P=(V,\leq)$ be an ordered set.  An initial segment $I$ is \emph{principal} if there exists $x\in V$ such that $I=\mydownarrow\,{x} :=\{y\in V: y \leq x\}$. A nonempty initial segment $I$ is an \emph{ideal} if $I$ is \emph{updirected}, that is for every $x,y\in I$ there exists $z\in I$ such that $x\leq z$ and $y\leq z$. For instance, every principal initial segment is an ideal.

In any ordered set, every initial segment is a union of ideals. The structure of initial segments is a bit simpler in the case of interval orders. If $I$ is an initial segment not generated by the set $\max(I)$ of the maximal elements of $I$, that is $I\neq \mydownarrow \max(I)$, then $I\setminus \mydownarrow \max(I)$ is updirected with no largest element (Lemma \ref{lem:is2}). Next, every ideal contains a cofinal chain (Theorem \ref{thm:cofinal-chain}). Furthermore, if an ideal has uncountable cofinality \(\kappa\) and there is no infinite antichain, then $I$ is the linear sum of ordered sets indexed by a chain of order type $\kappa$. From which follows that an uncountable scattered interval order with no infinite antichain cannot be prime (Proposition \ref{prop:notprime}).

\begin{lemma}\label{lem:is2}Let $P$ be an interval order. Let $I$ be an initial segment in $P$ such that $I\neq \mydownarrow \max(I)$. Then $I\setminus \mydownarrow \max(I)$ is updirected with no largest element. Furthermore, $I$ is the union of an antichain and a nonprincipal ideal.
\end{lemma}
\begin{proof}Let \(S:=I\setminus \mydownarrow\max(I)\). We first prove that \(S\) is updirected. Let \(x,y\in S\). Since
\(x\notin \mydownarrow\max(I)\), \(x\) is not maximal in \(I\). Hence, there exists \(x'\in I\) such that \(x<x'\). Similarly, there exists \(y'\in I\) such that \(y<y'\). By Lemma~\ref{lem:intervalorder-prop}(iv), \(U(x)\) and \(U(y)\) are comparable under set inclusion. Suppose, without loss of generality, that
\(U(x)\subseteq U(y)\). Since \(x'\in U(x)\), we get \(x'\in U(y)\), so \(y<x'\). Therefore \(x'\) is a common upper bound of \(x\) and \(y\).

It remains to check that \(x'\in S\). Since \(x'\in I\), it is enough to show that \(x'\notin\mydownarrow\max(I)\). If \(x'\leq m\) for some
\(m\in\max(I)\), then \(x<x'\leq m\), so \(x\in\mydownarrow\max(I)\), contradicting \(x\in S\). Hence, \(x'\in S\). Thus, \(S\) is updirected.

Next we prove that $S$ has no largest element. Suppose, by way of a contradiction, that $S$ has a largest element \(x\).  Then, on one hand, $x$ is not below or equal to any maximal element of $I$, but, then, on the other hand, no other elements of $I$ is strictly above $x$, so $x$ is maximal in $I$, a contradiction.

Since $\mydownarrow{S}$ is a nonprincipal initial segment and $S$ is updirected it follows readily that $\mydownarrow{S}$ is a nonprincipal ideal.

We now prove that for each $x\in \max(I)$, $D(x)\subseteq \mydownarrow {S}$. Indeed, let $t<x$ and let $t'\in S$. Since $S$ has no largest element we infer that there exists $u\in S$ such that $t'<u$. By Lemma~\ref{lem:intervalorder-prop}(iv), $U(t)$ and $U(t')$ are comparable with respect to set inclusion. Since $t<x$ and $t'\nless x$ we infer that $U(t')\subseteq U(t)$. Hence, $t<u$ and therefore $t\in \mydownarrow {S}$. This proves $D(x)\subseteq \mydownarrow {S}$.

Hence, $\max(I)=I\setminus \mydownarrow {S}$, that is, $I=\max(I)\cup\mydownarrow {S}$. Hence, if  $I\neq \mydownarrow \max(I)$, then $I$ is the union of an antichain, $\max(I)$, and a nonprincipal ideal, $\mydownarrow (I\setminus \mydownarrow \max(I))$.
\end{proof}

\begin{corollary}\label{cor:ideal-nomax}Let $P$ be an interval order and $I$ be a nonempty initial segment of $P$. If $\max(I)=\varnothing$, then $I$ is a nonprincipal ideal.
\end{corollary}

\begin{theorem}\label{thm:cofinal-chain}Every ideal in an interval order contains a cofinal chain.
\end{theorem}

To prove Theorem \ref{thm:cofinal-chain}, we will need the following result of Iwamura \cite{iwamura}.

\begin{lemma}\label{lem:iwamura} Let $P$ be an updirected ordered set of cofinality $\kappa$.  Then one can write $P$ as the union of a strictly increasing sequence $(I_\alpha)_{\alpha<\kappa}$ of ideals.
\end{lemma}

Lemma \ref{lem:iwamura} can be proved using the following result which is due to Krasner \cite{krasner} (see also \cite{delorme-pouzet}).

\begin{theorem}If $P$ up directed and $cf(P)=\kappa$ infinite, then there is an order preserving map $f: [\kappa]^{<\omega}\rightarrow P$ such that the image is cofinal in $P$.
\end{theorem}

To prove Theorem \ref{thm:cofinal-chain}, we will also need the following lemma.

\begin{lemma}\label{lem:is3}Let $P$ be an interval order, $I$ and $J$ be two ideals of $P$. If $I$ and $J$ are not principal, then $I$ and $J$ are comparable with respect to set inclusion. Moreover, if $I\subsetneq J$, then there exists $a\in J\setminus I$ such that $I \subseteq \mydownarrow a$.
\end{lemma}
\begin{proof}Consider $K=I\cup J$. Then $K$ is an initial segment having no maximal elements. It follows from  Corollary \ref{cor:ideal-nomax} that $K$ is an ideal and is in particular updirected. It follows that $I$ and $J$ are comparable with respect to set inclusion. Indeed, if not, then $I\setminus J\neq \varnothing \neq  J\setminus I$ and let $i\in I\setminus J$ and $j\in J\setminus I$. Because $K$ is an ideal there exists $k\in K$ above $i$ and $j$. Since $K=I\cup J$ we infer that $k\in I$ or $k\in J$ leading to $j\in I$ or $i\in J$. A contradiction.

For the moreover, suppose \(I\subsetneq J\), and choose \(v\in J\setminus I\). Since \(J\) is nonprincipal, \(J\) has no largest element. Hence, there exists \(v'\in J\) such that \(v<v'\). Since \(v\notin I\) and \(I\) is an initial segment, we must also have \(v'\notin I\).

We claim that \(I\subseteq \mydownarrow v'\). Suppose not. Then there exists \(u\in I\) such that \(u\nleq v'\). Since \(I\) is nonprincipal, it has no largest element. Hence, there exists \(u'\in I\) such that \(u<u'\). Because \(v\notin I\) and \(u'\in I\), we cannot have \(v\leq u'\). Hence, \(v\) and \(u'\) are incomparable.

Similarly, since \(u\nleq v'\) and \(u\in I\subseteq J\), while \(v'\in J\),
we cannot have \(v'\leq u\), for otherwise \(v'\in I\) by initiality of \(I\),
contrary to \(v'\notin I\). Hence, \(u\) and \(v'\) are incomparable.

Thus, \(u<u', v<v'\), and  \(u\) incomparable with \(v'\) and \(v\) incomparable with \(u'\).  Therefore \(\{u,u',v,v'\}\) induces a \(2\oplus2\), contradicting that \(P\) is an interval order.

Hence, \(I\subseteq \mydownarrow v'\). Taking \(a:=v'\) proves the required conclusion.
\end{proof}

\begin{proof}(Of Theorem \ref{thm:cofinal-chain})
 Let $I$ be an ideal in an interval order. Let $\kappa=cf(I)$. If $cf(I)=1$, then $I$ has a largest element. If $cf(I)=\omega$, then there exists a cofinal chain of order type $\omega$ \cite{milner-pouzet}. For larger cofinalities, we use Lemma \ref{lem:iwamura}.  Let $(I_\alpha)_\alpha<\kappa$ be a strictly increasing sequence of ideals whose union is $I$. We prove that there exists \(x_\alpha\in I_{\alpha+1}\setminus I_\alpha\) such that \(I_\alpha\subseteq \mydownarrow x_\alpha\).

  Suppose \(I_{\alpha+1}\) is principal, that is, \(I_{\alpha+1}=\mydownarrow p\) for some \(p\in I_{\alpha+1}\). Then necessarily \(p\not\in I_{\alpha}\). Indeed, suppose that \(p\in I_\alpha\). Since \(I_\alpha\) is an initial segment, we have
\(
\mydownarrow p\subseteq I_\alpha.
\)
But
\(
\mydownarrow p=I_{\alpha+1},
\)
and therefore
\(
I_{\alpha+1}\subseteq I_\alpha,
\)
contradicting
\(
I_\alpha\subsetneq I_{\alpha+1}.
\)
Hence, \(p\notin I_\alpha\). Thus, \(p\in I_{\alpha+1}\setminus I_{\alpha}\) and \(I_{\alpha}\subseteq I_{\alpha+1}= \mydownarrow p\).  Set \(x_\alpha:=p\). Next we suppose that \(I_{\alpha+1}\) is not principal.

 Suppose \(I_{\alpha}\) is principal, that is, \(I_{\alpha}=\mydownarrow p\) for some \(p\in I_{\alpha}\). Choose \(q\in I_{\alpha+1}\setminus I_{\alpha}\). Since \(I_{\alpha+1}\) is updirected, there exists \(r\in I_{\alpha+1}\) such that \(p,q\leq r\). Then clearly \(r\not \in I_{\alpha}\) and \(I_{\alpha}\subseteq \mydownarrow r\).  Set \(x_\alpha:=r\). Next we suppose that \(I_{\alpha}\) is not principal.

We are left with the case both \(I_{\alpha}\) and \(I_{\alpha+1}\) are nonprincipal. By Lemma \ref{lem:is3} applied to \(I_\alpha\subsetneq I_{\alpha+1}\),  there exists \(x_\alpha\in I_{\alpha+1}\setminus I_\alpha\) such that \(I_\alpha\subseteq \mydownarrow x_\alpha\).

We now show that the sequence \((x_\alpha)_{\alpha<\kappa}\) is increasing.
Let \(\alpha<\beta<\kappa\). By construction, \(x_\alpha\in I_{\alpha+1}\). Since the sequence \((I_\alpha)_{\alpha<\kappa}\) is increasing and
\(\alpha+1\leq \beta\), we have \(I_{\alpha+1}\subseteq I_\beta\). Also, by the choice of \(x_\beta\), \(I_\beta\subseteq \downarrow x_\beta\). Therefore, \(x_\alpha\in I_{\alpha+1}\subseteq I_\beta\subseteq \downarrow x_\beta\). Hence, \(x_\alpha\leq x_\beta\). Thus \((x_\alpha)_{\alpha<\kappa}\) is an increasing chain.

Finally, we prove that this chain is cofinal in \(I\). Let \(u\in I\). Since \(I=\cup_{\alpha<\kappa} I_\alpha\),
there exists \(\alpha<\kappa\) such that \(u\in I_\alpha\). By the choice of \(x_\alpha\), we have \(I_\alpha\subseteq \downarrow x_\alpha
\). Therefore, \(u\leq x_\alpha\). Hence, \((x_\alpha)_{\alpha<\kappa}\) is cofinal in \(I\).
\end{proof}

An ordered set $P$ is \emph{pure} if every proper initial segment $I$ has an upper bound in $P\setminus I$. This condition amounts to the fact that every noncofinal subset of P is strictly bounded above (indeed, if a subset $A$ of $P$ is not cofinal, then $\mydownarrow A \neq P$ hence from purity, $\mydownarrow A$, and thus $A$, is strictly bounded above. The converse is immediate).

We extract the following result from Theorem 24 of \cite{assous-pouzet}. 

\begin{theorem}\label{thm:assous-pouzet}Let $P$ be an ordered set with infinite uncountable cofinality $\nu$. Then $P$ is pure if and only if $P$ is a lexicographical sum $\sum_{\alpha\in K} P_\alpha$ where $K$ is a chain of order type $\nu$.
\end{theorem}

We will need the next three lemmas to prove that an uncountable scattered interval order with no infinite antichains cannot be prime.

\begin{lemma}\label{lem:ideal-pure}Let $P$ be an interval order with no infinite antichain. Then every ideal in $P$ is pure.
\end{lemma}
\begin{proof}
Let \(I\) be an ideal of \(P\). If \(I\) has a largest element, then \(I\) is
pure. Suppose therefore that \(I\) has no largest element.

Let \(J\subsetneq I\) be a proper initial segment. We show that \(J\) has an
upper bound in \(I\setminus J\).

First suppose that \(J=\mydownarrow \max(J)\). Since \(P\) has no infinite antichain, \(\max(J)\) is finite. Since \(I\) is
updirected, there exists \(b\in I\) such that
\(
m\leq b \text{ for every }m\in \max(J).
\)
Since \(I\) has no largest element, choose \(c\in I\) such that \(b<c\). Then \(c\notin J\), because if \(c\in J\), then \(c\leq m\) for some
\(m\in\max(J)\), while \(m\leq b<c\), a contradiction. Hence, \(c\in I\setminus J\) and  \(c\) is an upper bound of \(J\).

Now suppose that \(J\neq \mydownarrow \max(J)\).  By Lemma~\ref{lem:is2}, we may write \(J=A\cup K\), where \(A=\max(J)\) is an antichain and \(K\) is a nonprincipal ideal. Since \(P\) has no infinite antichain, \(A\) is finite. By Lemma~\ref{lem:is3}, applied to the nonprincipal ideals \(K\) and \(I\),
there exists \(a\in I\setminus K\) such that \(K\subseteq \mydownarrow a\). Now use the updirectedness of \(I\). Since \(A\) is finite and \(a\in I\),
there exists \(b\in I\) such that
\(
a\leq b
\quad\text{and}\quad
m\leq b\text{ for every }m\in A.
\)
Since \(I\) has no largest element, choose \(c\in I\) such that \(b<c\).  Then \(c\notin J\). Indeed, if \(c\in J\), then either \(c\in K\) or
\(c\in A\). If \(c\in K\), then \(c\leq a\leq b<c\), impossible. If \(c\in A=\max(J)\), then \(c\leq b<c\), again impossible. Thus, \(c\in I\setminus J\). Moreover, \(c\) is above every element of \(K\), since \(K\subseteq\mydownarrow a
\subseteq \mydownarrow c\), and \(c\) is above every element of \(A\). Therefore
\(c\) is an upper bound of \(J\) in \(I\setminus J\).

Hence, every proper initial segment \(J\) of \(I\) has an upper bound in \(I\setminus J\). Therefore \(I\) is pure.
\end{proof}

\begin{lemma}\label{lem:is4}Let $P=(V,\leq)$ be an interval order and $I$ be a nonprincipal ideal and
\[
I^+:=\{x\in V\setminus I:\exists y\in I,\ y<x\}
\]
Then $A:=V\setminus (I\cup I^+)$ is an antichain and every element of $A$ is incomparable with every element of $I$.
\end{lemma}
\begin{proof}
First we show that every element of \(A\) is incomparable with every element
of \(I\). Let \(a\in A\) and \(u\in I\). We cannot have \(a<u\), since
\(I\) is an initial segment and \(u\in I\); this would imply \(a\in I\),
contrary to \(a\in A\). We also cannot have \(u<a\), because then
\(a\in I^+\), again contrary to \(a\in A\). Hence, \(a\) is incomparable with \(u\).

It remains to prove that \(A\) is an antichain. Suppose, for a contradiction, that there exist \(v,v'\in A\) such that \(v<v'\).

Choose \(u\in I\). Since \(I\) is a nonprincipal ideal, it has no largest
element. Hence, there exists \(u'\in I\) such that \(u<u'\).

By the first paragraph, both \(v\) and \(v'\) are incomparable with every
element of \(I\). In particular, \(u\) and \(u'\) are both incomparable with \(v\) and \(v'\). Thus, the four elements \(u,u',v,v'\) induce a copy of \(2\oplus 2\), with \(u<u'\) and \(v<v'\). This contradicts the fact that \(P\) is an interval order. Therefore \(A\) is an antichain.
\end{proof}

\begin{lemma}\label{lem:notprime}
Let \(P=(V,\leq)\) be an interval order with no infinite antichains. If some
nonprincipal ideal \(I\) of \(P\) is a lexicographical sum
\[
I=\sum_{\alpha<\xi}K_\alpha,
\]
where \(\xi\) is a limit ordinal and \(K_\alpha\neq\varnothing\) for every \(\alpha<\xi\), then \(P\) is not prime.
\end{lemma}

\begin{proof}
Set
\[
I^+:=\{x\in V\setminus I:\exists y\in I,\ y<x\}
\]
and
\[
A:=V\setminus (I\cup I^+).
\]
By Lemma~\ref{lem:is4}, \(A\) is an antichain and every element of \(A\) is
incomparable with every element of \(I\).

We first prove that only finitely many elements of \(I^+\) are incomparable
with some element of \(I\).

Let \(v\in I^+\), and suppose that \(v\) is incomparable with some
\(u\in I\). Since \(I\) is a nonprincipal ideal, \(I\) has no largest element.
Hence, there exists \(u'\in I\) such that
\(
u<u'.
\)
Then \(u'\) is also incomparable with \(v\). Indeed, if \(u'<v\), then
\(u<v\), contradicting \(u\) is incomparable with \( v\). If \(v<u'\), then \(v\in I\), since
\(I\) is an initial segment and \(u'\in I\), contradicting
\(v\in I^+\subseteq V\setminus I\).

We claim that \(v\) is minimal in \(I^+\). Suppose not. Let
\(
v'\in I^+
\)
with
\(
v'<v.
\)
Since \(v'\notin I\), we have \(v'\nless u\) and \(v'\nless u'\), because
otherwise \(v'\in I\). Also, \(u\nless v'\), since otherwise \(u<v\), and
\(u'\nless v'\), since otherwise \(u'<v\). Therefore \(v'\) is incomparable
with both \(u\) and \(u'\).

Thus, \(
u<u'
\quad\text{and}\quad
v'<v
\)
form a copy of \({\bf 2}\oplus{\bf 2}\), contradicting that \(P\) is an
interval order. Hence, \(v\) is minimal in \(I^+\).

Therefore every element of \(I^+\) which is incomparable with some element of
\(I\) is minimal in \(I^+\). These elements form an antichain. Since \(P\) has
no infinite antichain, there are only finitely many of them. Let
\[
F:=\{v\in I^+: v\text{ is incomparable with some element of }I\}.
\]
Then \(F\) is finite. For each \(v\in F\), the set \(D(v)\cap I\) is a proper initial segment of
\(I\). We now choose a final interval of the index ordinal \(\xi\) which avoids all the sets
\(D(v)\cap I\), for \(v\in F\).

We first record an elementary fact. If \(J\) is an initial segment of \(I\) which is unbounded in
the index order \(\xi\), then \(J=I\).

Indeed, let \(u\in K_\beta\), where \(\beta<\xi\). Since \(J\) is unbounded in
the index order \(\xi\), there exists \(\alpha>\beta\) such that
\(
J\cap K_\alpha\neq\varnothing.
\)
Choose
\(
w\in J\cap K_\alpha.
\)
Since \(\beta<\alpha\), the lexicographical sum gives
\(
u<w.
\)
Because \(J\) is an initial segment of \(I\) and \(w\in J\), it follows that
\(
u\in J.
\)
Since \(\beta<\xi\) and \(u\in K_\beta\) were arbitrary, we conclude that
\(
J=I.
\)

Now fix \(v\in F\). The set
\(
D(v)\cap I
\)
is an initial segment of \(I\). Moreover, it is a proper initial segment. Indeed,
by the definition of \(F\), the element \(v\) is incomparable with at least one
element of \(I\). Hence \(v\) is not above all elements of \(I\), and therefore
\(
D(v)\cap I\neq I.
\)
By the elementary fact proved above, \(D(v)\cap I\) must be bounded in the index
order \(\xi\). Thus, for each \(v\in F\), there exists \(\beta_v<\xi\) such that
\(
D(v)\cap I
\subseteq
\cup_{\alpha<\beta_v}K_\alpha.
\)

Since \(F\) is finite, choose \(\beta<\xi\) such that
\(
\beta_v<\beta
\text{ for every }v\in F.
\)
Since \(\xi\) is a limit ordinal, we may choose \(\gamma\) such that
\(
\beta<\gamma<\xi.
\)
Set
\(
\lambda=[\gamma,\xi)
=
\{\alpha<\xi:\gamma\leq \alpha\}.
\)
Then \(\lambda\) is a nonempty proper final interval of \(\xi\). Moreover,
\(
\left(\cup_{\alpha\in\lambda}K_\alpha\right)\cap (D(v)\cap I)
=
\varnothing
\quad\text{for every }v\in F.
\) Setting
\(
M:=\cup_{\alpha\in\lambda}K_\alpha,
\)
we have
\(
M\cap (D(v)\cap I)=\emptyset \text{ for every }v\in F.
\)
Equivalently, every \(v\in F\) is incomparable with every element of \(M\).
We choose \(\lambda\) so that \(M\) is nontrivial and proper.

We now prove that \(M\) is a module of \(P\). Let \(z\in V\setminus M\).

If \(z\in I\setminus M\), then, since \(M\) is a final interval of the
lexicographical sum \(I\), \( z<m\) for every \(m\in M\).

If \(z\in A\), then \(z\) is incomparable with every element of \(I\), hence
with every element of \(M\).

Finally, suppose \(z\in I^+\). If \(z\in F\), then by the choice of \(M\),
\(z\) is incomparable with every element of \(M\). If \(z\notin F\), then
\(z\) is comparable with every element of \(I\). Since \(z\notin I\) and \(I\)
is an initial segment, \(z\) cannot be below any element of \(I\). Therefore
\(
m<z \text{ for every }m\in M.
\)

Thus, every element outside \(M\) is either below all elements of \(M\), above
all elements of \(M\), or incomparable with all elements of \(M\). Hence, \(M\)
is a nontrivial module of \(P\). This prove that \(P\) is not prime.
\end{proof}

\begin{proposition}\label{prop:notprime}An uncountable scattered unbounded interval order with no infinite antichain is not prime.
  \end{proposition}
  \begin{proof}
Let \(P=(V,\leq)\) be an uncountable scattered unbounded interval order with no
infinite antichain. It follows from Erd\H{o}s, Dushnik and Miller Theorem \cite{dushnik-miller} that $P$ contains either the first uncountable ordinal $\omega_1$, or its dual ${\omega_1}^*$.

We first assume that \(P\) contains a chain
\(
A=\{a_\alpha:\alpha<\omega_1\}
\) with
\(
a_\alpha<a_\beta \text{ whenever } \alpha<\beta.
\)
and let
\(
I:=\mydownarrow A
\). Then \(I\) is an ideal of \(P\). We next prove that \(I\) is nonprincipal. Suppose, to the contrary, that
\(
I=\mydownarrow p
\)
for some \(p\in P\). Since \(p\in I\), there exists \(\alpha<\omega_1\) such that
\(
p\leq a_\alpha.
\)
But \(a_{\alpha+1}\in I=\mydownarrow p\), so
\(
a_{\alpha+1}\leq p\leq a_\alpha,
\)
contradicting
\(
a_\alpha<a_{\alpha+1}.
\)
Thus, \(I\) is nonprincipal.

We now claim that \(I\) has cofinality \(\omega_1\). The chain \(A\) is cofinal
in \(I\), so
\(
\operatorname{cf}(I)\leq \omega_1.
\)
On the other hand, \(I\) cannot have a countable cofinal subset. Indeed, suppose
\(
C=\{c_n:n\in\mathbb N\}
\)
were cofinal in \(I\). For each \(n\), choose \(\alpha_n<\omega_1\) such that
\(
c_n\leq a_{\alpha_n}.
\)
Let
\(
\beta>\sup_{n\in\mathbb N}\alpha_n
\)
with \(\beta<\omega_1\). Since \(C\) is cofinal in \(I\), there exists \(n\) such
that
\(
a_\beta\leq c_n.
\)
Therefore
\(
a_\beta\leq c_n\leq a_{\alpha_n},
\)
contradicting \(\alpha_n<\beta\). Hence, \(I\) has no countable cofinal subset, and so
\(
\operatorname{cf}(I)=\omega_1.
\)

By Lemma~\ref{lem:ideal-pure}, every ideal of an interval order with no infinite
antichain is pure. Hence, \(I\) is a pure nonprincipal ideal of cofinality \(\omega_1\).

By Theorem \ref{thm:assous-pouzet}, \(I\) can be written as a lexicographical sum
\(
I=\sum_{\alpha<\omega_1}K_\alpha,
\)
where each \(K_\alpha\) is nonempty. Since \(\omega_1\) is a limit ordinal, we may
apply Lemma~\ref{lem:notprime}. Therefore \(P\) is not prime.

If $P$ contains ${\omega_1}^*$ apply the previous argument to the dual of $P$.
\end{proof}

\subsection{Prime interval graphs with no an infinite clique are at most countable}

A \textit{graph} is a pair $G:= (V, E)$, where $E$ is a subset of $[V]^2$, the set of $2$-element subsets of $V$. Elements of $V$ are the \textit{vertices} of $G$ and elements of $ E$ its \textit{edges}. The \textit{complement} of $G$ is the graph $G^c$ whose vertex set is $V$ and edge set ${\overline { E}}:= [V]^2\setminus  E$. The \emph{comparability graph}, respectively the \emph{incomparability graph},
of an ordered set $P= (V,\leq)$ is the (undirected) graph, denoted by $\comp(P)$, respectively $\incg(P)$, with vertex set $V$ and edges the pairs
$\{u,v\}$ of comparable distinct vertices (that is, either $u< v$ or $v<u$) respectively incomparable vertices. A graph $G=  (V, E)$ is a
\emph{comparability graph} if its edge set is the set of comparabilities of some order on $V$.  A graph is called an \emph{interval graph} if each of its vertices can be associated with a nonempty interval on a chain in such a way that two vertices are adjacent in the graph if and only if the associated intervals have a nonempty intersection. An interval graph is the incomparability graph of an interval order (see \cite{gilmore}).

The following is due to Gallai \cite{gallai} in the finite case and to Kelly \cite{kelly85} in the infinite case.

\begin{theorem}\label{thm:kelly}
An ordered set is prime if and only if its comparability graph is prime.
\end{theorem}

Since a graph is prime if and only if its complement is prime, it follows that:

\begin{corollary}Every prime interval graph with no infinite clique is at most countable.
\end{corollary}

\section{Modular decomposition of graphs with no infinite independent sets}\label{section:tree}

Let $R=(V,\rho)$ be a binary relation. A module $A$ is \emph{strong} if it is nonempty and every other module is disjoint from $A$ or is comparable to $A$ with respect to set inclusion. The singletons in $V$ and the set $V$ are (trivial) strong modules. The modules of a chain are its intervals; in particular the strong modules of a chain are trivial. It may happen that a binary relation has no maximal proper strong modules (see Example 4.3 from \cite{courcelle-delhomme}). The notion of strong module was introduced by Gallai \cite{gallai} who showed that \emph{a finite ordered set and its comparability graph have the same strong modules}, from which it follows that \emph{an ordered set is prime if and only if its comparability graph is prime}. These results were extended to infinite comparability graphs by Kelly \cite{kelly85}.

Following Courcelle and Delhomm\'e \cite{courcelle-delhomme}, for any nonempty $A \subseteq V$, we let $S(A)$ denote the intersection of all strong modules including $A$; thus $S(A)$ is the smallest strong module including $A$. For possibly equal vertices $x$ and $y$ of a binary relation $R$, let $S(x, y)$ denote the least strong module $S(\{x,y\})$ containing $x$ and $y$. We call \emph{robust} any such strong module. Notice that a module is robust if and only if it is of the form $S(F)$ for some nonempty finite set $F$ of vertices (see Lemma 3.1(4) of \cite{courcelle-delhomme}). For an example, if $A$ is a module and $R_{\restriction A}$ is prime, then $A$ is robust. For finite binary relations, the notions of a strong module and of a robust module coincide.

\begin{df}
Let $R=(V,\rho)$ be a binary relation. Let $M$ be a strong module in $R$. Let $x,y \in M$. We set $x \equiv_M y$ if either $x=y$ or there is a strong module containing $x$ and $y$ and properly contained in $M$.
\end{df}

The following is Lemma 6.5 from \cite{hahn-pouzet-woodrow}.

\begin{lemma}
Let $R=(V,\rho)$ be a binary relation. Let $M$ be a strong module in $R$. The relation $\equiv_M$ is an equivalence relation on $M$ whose equivalence classes are strong modules. If $M$ has more than one element, then there are at least two equivalence classes if and only if $M$ is robust. Furthermore, these classes are the maximal strong modules properly included in $M$.
\end{lemma}

We call \emph{components} of a robust module $M$ the equivalence classes of the relation $\equiv_M$. The components need not be robust. If $M$ is a robust module with at least two elements of a binary relation $R$, then for two distinct components $I$, $J$ in $M$, $(x,y)\in \rho$ for $x\in I$ and $y \in J$ depends only upon $I$ and $J$. Hence, the binary relation  on $M$ induces a binary relation $R/\equiv_M$ on the set $M/\equiv_M$ of components of $M$, called the \emph{Gallai quotient} of $M$, and $R_{\restriction M}$ is the lexicographical sum of the components indexed by $R_{\restriction M}/\equiv_M$. This quotient $R_{\restriction M}/\equiv_M$ has a special structure: its strong modules are trivial.
The following theorem is the central result of the decomposition theory of binary relation. It describes the structure of the Gallai quotient. It is due to Gallai [14] for finite graphs, to Ehrenfeucht and Rozenberg \cite{EhrenfeuchtRozenbergII} for finite binary relations and to Harju and Rozenberg \cite{HarjuRozenberg} for infinite binary relations (see also \cite{courcelle-delhomme} and \cite{boussairi}).

\begin{theorem}\label{thm:no-strong}A binary relation $R$ with at least two vertices has no nontrivial strong modules if and only if $R$ is prime, or is edge-free, or is complete or is a linear order.
\end{theorem}

Specializing to the case of posets we have,

\begin{corollary}\label{cor:no-strong}A poset $P$ with at least two vertices has no nontrivial strong modules if and only if $P$ is prime, an antichain or is a linear order.
\end{corollary}

\begin{theorem}\label{thm:robust-lexico}
If $R$ is a binary relation, then every robust module with at least two elements is the lexicographic sum of its components and the quotient is either prime, or is edge-free, or is complete or is a linear order.
\end{theorem}

Specializing to the case of posets we have,

\begin{corollary}\label{cor:robust-lexico}
If $P$ is a poset, then every robust module with at least two elements is the lexicographic sum of its components and the quotient is either prime, or is an antichain, or is a chain.
\end{corollary}

\begin{lemma}\label{lem:maximal-independent}Let $G=(V,E)$ be a graph. Let $M$ be a module in $G$. If $M$ contains a maximal (with respect to set inclusion) independent set in $G$, then $V\setminus M$ is a module in $G$.
\end{lemma}
\begin{proof}Let $M$ be a module in $G$. Let $I$ be a maximal independent set in $G$ and suppose $I\subseteq M$. Let $v\not \in M$. Then $v\not \in I$. From the maximality of $I$ we infer that there exists $i\in I$ such that $\{v,i\}\in E$. Since $I\subset M$ and $M$ is a module we infer that for all $m\in M$, $\{v,m\}\in E$. Hence, every element not in $M$ is adjacent to all elements in $M$. This proves that  $V\setminus M$ is a module in  $G$.
\end{proof}

\begin{lemma}\label{lem:strong-robust}Let $G=(V,E)$ be a graph with no infinite independent set or no infinite clique. Then every strong module is robust.
\end{lemma}
\begin{proof}Since a graph and its complement have the same set of strong modules, we may assume without loss of generality that $G$ has no infinite independent set. Let $M$ be a strong module in $G$. Let $I$ be a maximal independent set in $M$. Since $G$ has no infinite independent set, $I$ must be finite. Then $S(I)$ is a robust module included in $M$. If $S(I)=M$, then we are done. Next we suppose that $S(I)\neq M$. It follows from Lemma \ref{lem:maximal-independent} applied to $G_{\restriction M}$ that $M\setminus S(I)$ is a nonempty module in $G_{\restriction M}$. Let $x\in S(I)$ and $y\in M\setminus S(I)$ and consider $S(x,y)$. This is a strong module included in $M$  that intersects the two modules $S(I)$ and $M\setminus S(I)$. Since $y\not \in S(I)$ and $x\not \in M\setminus S(I)$ we infer that $M\subseteq S(x,y)$. Hence, $M= S(x,y)$. This proves that $M$ is robust as required.
\end{proof}

The following is now a consequence of Lemma \ref{lem:strong-robust} and Theorem \ref{thm:robust-lexico}.

\begin{corollary}If $G=(V,E)$ is a  graph with no infinite independent set or no infinite clique, then $G$ has a Gallai decomposition.
\end{corollary}

\subsection{Gallai decomposition of interval orders with no infinite antichains}\label{subsection:interval-gallai}

Let $P=(V,\leq)$ be an ordered set. We recall that $\sing(P)$ is the set of singular vertices of $P$, that is the set of  vertices that belong to a unique maximal antichain in $P$.

\begin{lemma}\label{lem:singular}Let $P=(V,\leq)$ be an interval order distinct from an antichain. Then \break $(\sing(P),\leq)$ is a linear sum of antichains. If $P$ is prime, then $(\sing(P),\leq)$ is a chain.
\end{lemma}
\begin{proof}
Since a singular vertex belongs to a unique maximal antichain, two distinct
singular vertices are either contained in the same maximal antichain or are
comparable. Thus, the incomparability classes of \(\sing(P)\) are antichains,
and these classes are linearly ordered. Hence, \((\sing(P),\leq)\) is a linear
sum of antichains.

Now suppose that \(u,v\in\sing(P)\) are distinct and incomparable. Let \(A\)
be the unique maximal antichain containing \(u\) and \(v\). We prove that
\(\{u,v\}\) is a module.

Let \(x\notin\{u,v\}\). If \(x\in A\), then \(x\) is incomparable with both
\(u\) and \(v\). Suppose now that \(x\notin A\). Since \(u\) and \(v\) are
singular and \(A\) is their unique maximal antichain, \(x\) must be comparable
with both \(u\) and \(v\). Moreover, \(x\) cannot be below one of \(u,v\) and
above the other, because \(u\) is incomparable with \(v\). Therefore either \(x\) is below \(u\) and \(v\) or \(x\) is above \(u\) and \(v\). Hence, \(\{u,v\}\) is a nontrivial module of \(P\).

If \(P\) is prime, no such pair \(u,v\) can exist. Therefore every two singular vertices are comparable, and so \((\sing(P),\leq)\) is a chain.
\end{proof}

\begin{lemma}\label{lem:lexico-sum}Let $(P_i)_{i\in Q}$ be a family of ordered sets indexed by an ordered set $Q$. Then the lexicographical sum $ \sum_{i\in Q} P_{i}$ is an interval order if and only if $Q$ and all the $P_i$'s are interval orders and for all $i\not \in \sing(Q)$, $P_i$ is an antichain  and, for every maximal antichain $A\subseteq \sing(Q)$, there is at most one $a\in A$ such that $P_a$ is not an antichain.
\end{lemma}

\begin{proof} Suppose first that \(P:=\sum_{i\in Q}P_i\) is an interval order. Then each \(P_i\) is an induced
subposet of \(P\), hence each \(P_i\) is an interval order. Also, choosing one
element \(x_i\in P_i\) for each \(i\in Q\), the induced subposet on
\(\{x_i:i\in Q\}\) is isomorphic to \(Q\). Hence, \(Q\) is an interval order.

Let \(i\notin\sing(Q)\). If \(P_i\) were not an antichain, choose
\(a<b\) in \(P_i\). Since \(i\notin\sing(Q)\), the set of elements of \(Q\)
incomparable with \(i\) is not an antichain. Hence, there exist
\(j<k\) in \(Q\), both incomparable with \(i\). Choose
\(c\in P_j\) and \(d\in P_k\). Then, in \(P\),
\(
a<b,\quad c<d,
\)
and all cross pairs are incomparable. Thus, \(P\) contains a \(2\oplus2\),
contradicting that \(P\) is an interval order. Therefore \(P_i\) is an
antichain.

Now let \(A\) be a maximal antichain of \(\sing(Q)\). If two distinct
\(i,j\in A\) had \(P_i\) and \(P_j\) both non-antichains, choose
\(
a_i<b_i\quad\text{in }P_i,
\quad
a_j<b_j\quad\text{in }P_j.
\)
Since \(i\) is incomparable with  \(j\), these two chains form a \(2\oplus2\) in \(P\), again a
contradiction. Hence, at most one \(a\in A\) has \(P_a\) non-antichain.

Conversely, assume that \(Q\) and all \(P_i\)'s are interval orders, that
\(P_i\) is an antichain whenever \(i\notin\sing(Q)\), and that every maximal
antichain of \(\sing(Q)\) contains at most one index \(a\) with \(P_a\) not an
antichain.

Suppose, toward a contradiction, that \(P\) contains an induced
\(2\oplus2\), say
\(
x_1<x_2,\quad y_1<y_2,
\)
with all cross pairs incomparable. Let \(i\) be the index of the component
containing \(x_1,x_2\), if the first chain lies inside one component; otherwise
let \(i_1<i_2\) be the indices of \(x_1,x_2\). Define similarly for
\(y_1,y_2\).

If both chains use two different components, then their indices give a
\(2\oplus2\) in \(Q\), contradicting that \(Q\) is an interval order.

If one chain lies inside a component, say \(x_1,x_2\in P_i\), and the other chain uses two different components \(P_j,P_k\) with \(j<k\),
then \(i\) is incomparable with both \(j\) and \(k\). Since \(j<k\), the set of
elements incomparable with \(i\) is not an antichain. Thus, \(i\notin\sing(Q)\). But \(P_i\) contains the chain \(x_1<x_2\), so \(P_i\) is not an antichain,
contradicting the hypothesis.

Therefore both chains must lie inside components: \(x_1,x_2\in P_i\) and \(y_1,y_2\in P_j\). Since the two chains are mutually incomparable in \(P\), we have \(i\) is incomparable with \(j\) in \(Q\). Also \(P_i\) and \(P_j\) are not antichains, so by hypothesis \(i,j\in\sing(Q)\).
By Lemma~\ref{lem:singular}, \((\sing(Q),\leq)\) is a linear sum of antichains.
Hence, two incomparable singular elements belong to the same maximal antichain
of \(\sing(Q)\). This maximal antichain contains both \(i\) and \(j\), and both
\(P_i\) and \(P_j\) are non-antichains, contradicting the hypothesis.

Thus, \(P\) contains no induced \(2\oplus2\). Hence, \(P\) is an interval order.
\end{proof}

We are now ready to proceed with the proof of Theorem \ref{thm:main-lex}. Let $P$ be an interval order. If $P$ has no nontrivial strong modules, then it follows from Corollary \ref{cor:no-strong} that $P$ is either prime, an antichain or is a linear order. Now suppose $P$ has a nontrivial strong module. Suppose further that $P$ has no infinite antichains. It follows from Corollary \ref{cor:robust-lexico} that $P$ is a lexicographical sum of its maximal strong modules over a prime interval, an antichain or  a chain. The conclusion follows from Lemma \ref{lem:lexico-sum}.

\section{Examples of well-quasi-ordered prime interval orders whose chain of maximal antichains has arbitrary countable Hausdorff rank}\label{section:example}

A characterization of scattered linear orders was obtained by Hausdorff
\cite{hausdorff} in 1908 and rediscovered by Erd\H{o}s and Hajnal
\cite{eh63} in 1962.

\begin{theorem}[Hausdorff's classification theorem]
\label{thm:hausdorff-scattered}
Let \(\mathfrak{S}\) be the smallest class of chains satisfying the following
conditions:
\begin{itemize}
  \item The one-element chain belongs to \(\mathfrak{S}\).
  \item If \(C,C'\in \mathfrak{S}\), then \(C\oplus C'\in \mathfrak{S}\).
  \item If \(\kappa\) is a regular cardinal and
  \(\{C_\gamma : \gamma<\kappa\}\subseteq \mathfrak{S}\), then both
  \[
  \oplus_{\gamma<\kappa}C_\gamma
  \qquad\text{and}\qquad
  \oplus_{\gamma<\kappa^*}C_\gamma
  \]
  belong to \(\mathfrak{S}\), where \(\kappa^*\) denotes the dual of
  \(\kappa\).
\end{itemize}
Then \(\mathfrak{S}\) is precisely the class of scattered chains.
\end{theorem}

The following theorem expresses the decomposition of scattered chains into sums
of chains of strictly smaller Hausdorff rank (see \cite{fraissetr},
Section~6.2.6, p.~171).

\begin{theorem}\label{thm:hausdorff-rank-scattered}
Every scattered chain is either:
\begin{enumerate}[$(1)$]
\item a finite linear sum of chains of strictly smaller rank;
\item a linear sum of chains of strictly smaller rank indexed by a regular
cardinal;
\item a linear sum of chains of strictly smaller rank indexed by the dual of a
regular cardinal.
\end{enumerate}
\end{theorem}

For countable scattered chains, Theorem
\ref{thm:hausdorff-rank-scattered} takes the following form.

\begin{theorem}\label{thm:countable-hausdorff-rank-scattered}
Every countable scattered chain is either a finite linear sum of chains of
strictly smaller rank, or a linear sum of chains of strictly smaller rank
indexed by \(\omega\) or by \(\omega^*\).
\end{theorem}

We will use Theorem~\ref{thm:countable-hausdorff-rank-scattered}, together
with induction on the Hausdorff rank, to construct examples of
well-quasi-ordered prime interval orders whose chains of maximal antichains
have arbitrary countable Hausdorff rank.

\begin{figure}[ht]
\begin{center}
\leavevmode \epsfxsize=4.5in \epsfbox{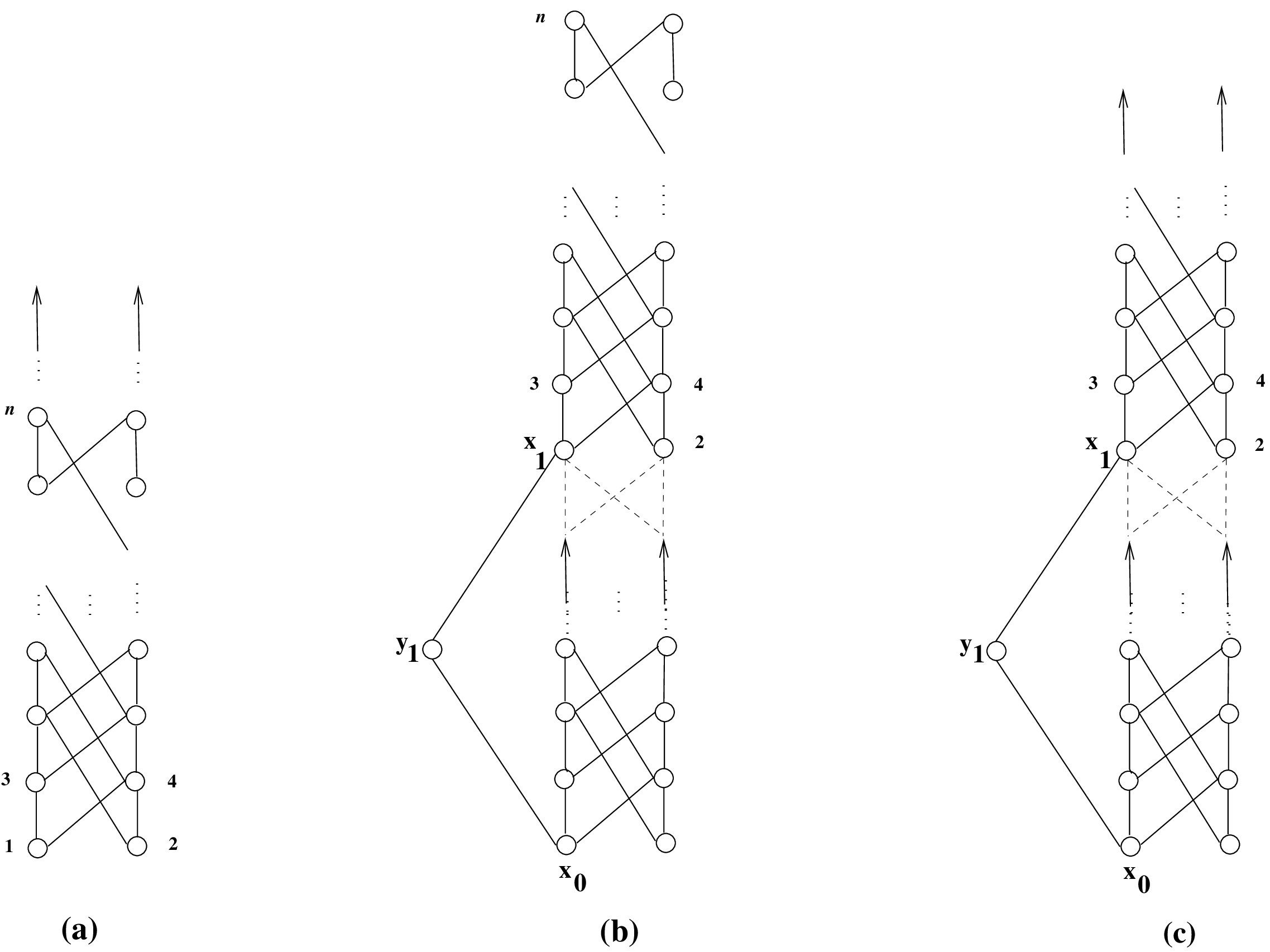}
\end{center}
\caption{The dashed lines in (b) and (c) indicate linear sum.\\
(a) A well-quasi-ordered prime interval order whose chain of maximal antichains has order type $\omega$. The incomparability graph is an infinite one-way path. (b) A well-quasi-ordered prime interval order whose chain of maximal antichains has order type $\omega+n-1$, for $n\in \NN$. (c) A well-quasi-ordered prime interval order whose chain of maximal antichains has order type $\omega\oplus \omega$.}
\label{examples}
\end{figure}

A \emph{path} is a graph \(P\) for which there exists a one-to-one map
\[
f:V(P)\rightarrow I,
\]
where \(I\) is an interval of the chain \(\ZZ\), such that
\[
\{u,v\}\in E(P)
\quad\Longleftrightarrow\quad
|f(u)-f(v)|=1
\]
for all \(u,v\in V(P)\).

If \(I=\{1,\dots,n\}\), then the path is denoted by \(P_n\). Its length is
\(n-1\). Thus, \(P_2\) consists of a single edge, while \(P_1\) consists of a
single vertex. We denote by \(P_{\NN}\) the one-way infinite path obtained by
taking \(I=\NN\).

For every integer \(n\geq 4\), the path \(P_n\) is prime and its complement
\(\overline{P_n}\) is a prime comparability graph. Consequently,
\(\overline{P_n}\) has exactly two transitive orientations. These two
orientations are isomorphic; we denote either of them by \(I_n\). The ordered
set \(I_n\) is a prime interval order whose chain of maximal antichains has
cardinality \(n-1\) and Hausdorff rank \(0\).

Similarly, the one-way infinite path \(P_{\NN}\) is prime and its complement
\(\overline{P_{\NN}}\) is a prime comparability graph. Hence, \(\overline{P_{\NN}}\) has exactly two transitive orientations. One of these
orientations, denoted by \(I_{\NN}\), is a prime well-quasi-ordered interval
order whose chain of maximal antichains has order type \(\omega\) (see Figure
\ref{examples}(a)) and Hausdorff rank \(0\).

The following basic construction will allow us to obtain well-quasi-ordered prime interval orders of arbitrary countable Hausdorff rank.

Let \(\kappa\geq 2\) be either a finite integer or \(\kappa=\omega\). Let
\(
(P_i=(V_i,\leq_i))_{i<\kappa}
\)
be a sequence of prime well-quasi-ordered interval orders such that
\(
|V_i|\geq 4\) for every \(i<\kappa\).

For each \(i<\kappa\), choose
\(
x_i\in \min(P_i)
\)
such that \(U(x_i)\) is maximal, with respect to set inclusion, among the sets
\(
\{U(t): t\in \min(P_i)\}.
\)

For each \(1\leq i<\kappa\), let \(y_i\) be a new element not belonging to
\(
\cup_{i<\kappa}V_i.
\)
Set
\[
V:=(\cup_{i<\kappa}V_i)\cup\{y_i:1\leq i<\kappa\}.
\]

We now define an ordered set \(Q=(V,\leq_Q)\). First, the restriction of
\(\leq_Q\) to \(\cup_{i<\kappa}V_i\) is the linear sum
\(
\sum_{i<\kappa}P_i.
\)
Next, for every \(1\leq i<\kappa\), we impose
\(
x_{i-1}<_{Q} y_{i}<_{Q} x_i.
\)
The order relation \(\leq_Q\) is then obtained by taking the transitive
closure.

Observe that
\(
y_1<_{Q}y_2<_{Q}y_3<_{Q}\cdots,
\)
so the set \(\{y_i:1 \leq i<\kappa\}\) forms a well-founded chain in \(Q\).

See Figure~\ref{construction} for a depiction of the ordered set \(Q\). The
ordered set shown in Figure~\ref{examples}(c) is obtained from this
construction in the case \(\kappa=2\) and \(P_i=I_{\NN}\) for all
\(i<2\).

\begin{figure}[ht]
\begin{center}
\leavevmode \epsfxsize=1.7in \epsfbox{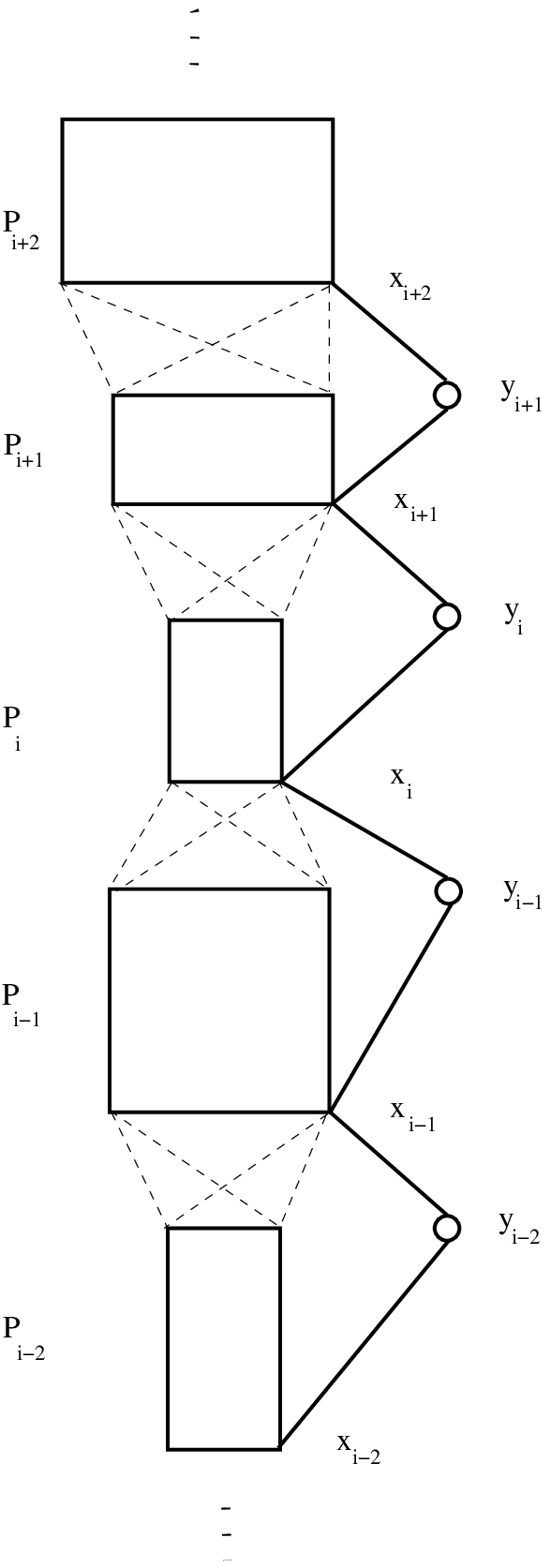}
\end{center}
\caption{The ordered set $Q$. The $P_j$'s are well-quasi-ordered prime interval orders. The dashed lines indicate linear sum.} \label{construction}
\end{figure}

\begin{lemma}\label{lem:basi-construction}
The ordered set \(Q\) is a well-quasi-ordered prime interval order. Put
\[
 C_i:=AM(P_i)_\leq \qquad (i<\kappa),
\]
and let \(\mathbf 1\) denote the one-element chain. For
\(1\leq i<\kappa\), put
\[
 \widehat C_i:=
 \begin{cases}
  C_i,&\text{if \(C_i\) is infinite},\\
  \mathbf 1\oplus C_i,&\text{if \(C_i\) is finite}.
 \end{cases}
\]
Then each \(C_i\) is a nonempty well-order and
\begin{equation}
 AM(Q)_\leq
 \cong
 C_0\oplus
 \bigoplus_{1\leq i<\kappa}\widehat C_i.
 \label{eq:basic-AM}
\end{equation}
In particular, if \(C_i\) is infinite for every
\(1\leq i<\kappa\), then
\begin{equation}
 AM(Q)_\leq
 \cong
 \sum_{i<\kappa}C_i
 =
 \sum_{i<\kappa}AM(P_i)_\leq.
 \label{eq:basic-AM-infinite}
\end{equation}
\end{lemma}

\begin{proof}
For \(i<\kappa\), set
\[
 M_i:=\min(P_i).
\]
Since \(P_i\) is well-quasi-ordered, \(M_i\) is finite. By
Lemma~\ref{lem:minimal-element},
\begin{equation}
 U_{P_i}(x_i)=V_i\setminus M_i.
 \label{eq:basic-Ux}
\end{equation}
In particular, \(x_i\) is below every nonminimal element of \(P_i\),
and \(M_i\) is the unique maximal antichain of \(P_i\) containing
\(x_i\).

Notice also that \(|M_i|\geq2\). Indeed, \(P_i\) being prime it cannot have a samllest element. 

We shall also use the following consequence of primality:
\begin{equation}
 M_i\setminus\{x_i\}
 \text{ is not a maximal antichain of }P_i-x_i.
 \label{eq:basic-notmax}
\end{equation}
Suppose otherwise. Then every element \(v\in V_i\setminus M_i\) is
comparable with some element of \(M_i\setminus\{x_i\}\), and, since
the latter elements are minimal, \(v\) is above such an element. The
sets
\[
 U_{P_i}(m),\qquad m\in M_i,
\]
are linearly ordered by inclusion because \(P_i\) is an interval
order. Since \(M_i\setminus\{x_i\}\) is finite, choose
\(m_i\in M_i\setminus\{x_i\}\) for which \(U_{P_i}(m_i)\) is maximal
among these upper sets. The assumed maximality of
\(M_i\setminus\{x_i\}\) in \(P_i-x_i\) gives
\[
 V_i\setminus M_i
 \subseteq
 \bigcup_{m\in M_i\setminus\{x_i\}}U_{P_i}(m)
 =
 U_{P_i}(m_i).
\]
The reverse inclusion is automatic, so, by (\ref{eq:basic-Ux}),
\[
 U_{P_i}(m_i)
 =
 V_i\setminus M_i
 =
 U_{P_i}(x_i).
\]
Since both \(m_i\) and \(x_i\) are minimal, they have the same strict
down-set, namely the empty set. Thus \(\{m_i,x_i\}\) is a nontrivial
module of \(P_i\), a contradiction. This proves
(\ref{eq:basic-notmax}).

\medskip
\noindent
\textbf{Claim 1: \(Q\) is well-quasi-ordered.}

The restriction of \(Q\) to
\(
 \bigcup_{i<\kappa}V_i
\)
is the linear sum
\(
 \sum_{i<\kappa}P_i,
\)
and hence is well-quasi-ordered. The elements \(y_i\),
\(1\leq i<\kappa\), form a well-ordered chain.

If an antichain contains \(y_i\), then it can meet at most one of
\(V_{i-1}\) and \(V_i\), and it meets no other \(V_j\). Hence it is
finite because each \(P_j\) has no infinite antichain. Thus \(Q\) has
no infinite antichain.

There is no infinite descending chain in \(Q\). Indeed, such a chain
can contain only finitely many of the \(y_i\)'s, because the \(y_i\)'s
form a well-ordered chain. After deleting a finite initial segment, it
would therefore be a descending chain in
\(\sum_{i<\kappa}P_i\). The indices of the blocks met by such a chain
are nonincreasing, hence eventually constant, yielding an infinite
descending chain in one of the \(P_i\)'s. This is impossible.
Therefore \(Q\) is well-quasi-ordered.

\medskip
\noindent
\textbf{Claim 2: We classify the maximal antichains of \(Q\) and
determine their order in \(AM(Q)_\leq\).}

For \(1\leq i<\kappa\), put
\begin{equation}
 B_i:=
 \{y_i\}\cup\bigl(M_i\setminus\{x_i\}\bigr).
 \label{eq:basic-Bi}
\end{equation}
By (\ref{eq:basic-Ux}), \(y_i\) is below \(x_i\) and every nonminimal
element of \(P_i\), and it is incomparable with every element of
\(M_i\setminus\{x_i\}\). Thus \(B_i\) is an antichain. It is maximal:
vertices in earlier blocks are below the vertices of
\(M_i\setminus\{x_i\}\), vertices in later blocks are above \(y_i\),
the vertex \(x_i\) and all nonminimal vertices of \(P_i\) are
comparable with \(y_i\), and every other \(y_j\) is comparable with
\(y_i\). Hence
\[
 B_i\in AM(Q).
\]

For \(i<\kappa\), define a family
\(\mathcal C_i\subseteq AM(Q)\) as follows. If \(i+1<\kappa\), let
\begin{equation}
 \mathcal C_i
 :=
 \{M_i\}
 \cup
 \bigl\{
 A\cup\{y_{i+1}\}:
 A\in AM(P_i),\ A\neq M_i
 \bigr\};
 \label{eq:basic-Ci}
\end{equation}
if \(\kappa=m<\omega\) and \(i=m-1\), let
\begin{equation}
 \mathcal C_{m-1}:=AM(P_{m-1}).
 \label{eq:basic-Clast}
\end{equation}

Every member of \(\mathcal C_i\) is a maximal antichain of \(Q\).
Indeed, if \(A\in AM(P_i)\) and \(A\neq M_i\), then \(x_i\notin A\),
because \(M_i\) is the unique maximal antichain containing \(x_i\).
The element \(y_{i+1}\) is incomparable with every element of
\(V_i\setminus\{x_i\}\), so
\[
 A\cup\{y_{i+1}\}
\]
is an antichain. Its maximality follows from the maximality of \(A\)
in \(P_i\). The set \(M_i\) itself is maximal in \(Q\), because every
\(y_j\) is comparable with \(x_i\). Finally, in the last block of a
finite construction, every maximal antichain of \(P_{m-1}\) remains
maximal in \(Q\): if it is not \(M_{m-1}\), it contains a nonminimal
element, which is above \(y_{m-1}\).

We now show that there are no other maximal antichains. Let
\(A\in AM(Q)\). Since the \(V_i\)'s form a linear sum, \(A\) meets at
most one \(V_i\), and since the \(y_i\)'s form a chain, \(A\) contains
at most one \(y_i\). Also, \(A\) cannot consist only of a vertex
\(y_i\), since \(|M_i|\geq2\) and \(B_i\) properly contains
\(\{y_i\}\).

Suppose first that \(A\) contains no \(y_i\). Then
\(A\subseteq V_i\) for some \(i\), and \(A\in AM(P_i)\). If
\(i+1<\kappa\) and \(A\neq M_i\), then \(x_i\notin A\) and
\(y_{i+1}\) is incomparable with every element of \(A\), contradicting
maximality in \(Q\). Thus in this case
\[
 A=M_i.
\]
If \(i\) is the last index of a finite construction, then
\(A\in\mathcal C_i\) by (\ref{eq:basic-Clast}).

Now suppose that \(y_j\in A\). The only blocks containing vertices
incomparable with \(y_j\) are \(V_{j-1}\) and \(V_j\), and an
antichain cannot meet both of these blocks.

If \(A\) meets \(V_j\), then the vertices of \(V_j\) incomparable with
\(y_j\) are exactly
\[
 M_j\setminus\{x_j\},
\]
and maximality gives
\[
 A=B_j.
\]

It remains to consider the case in which \(A\) meets \(V_{j-1}\).
Put
\[
 S:=A\cap V_{j-1}.
\]
Since \(y_j\) is incomparable with every element of
\(V_{j-1}\setminus\{x_{j-1}\}\), maximality of \(A\) implies that
\(S\) is a maximal antichain of \(P_{j-1}-x_{j-1}\).

We claim that \(S\) is a maximal antichain of \(P_{j-1}\). Otherwise,
the only vertex which could be adjoined to \(S\) is \(x_{j-1}\). Hence
every element of \(S\) is incomparable with \(x_{j-1}\), so, by
(\ref{eq:basic-Ux}),
\[
 S\subseteq M_{j-1}\setminus\{x_{j-1}\}.
\]
Maximality in \(P_{j-1}-x_{j-1}\) then forces
\[
 S=M_{j-1}\setminus\{x_{j-1}\},
\]
contradicting (\ref{eq:basic-notmax}). Therefore
\[
 S\in AM(P_{j-1}),
\]
and necessarily \(S\neq M_{j-1}\). Thus
\[
 A=S\cup\{y_j\}\in\mathcal C_{j-1}.
\]

This completes the classification:
\begin{equation}
 AM(Q)
 =
 \bigcup_{i<\kappa}\mathcal C_i
 \;\cup\;
 \{B_i:1\leq i<\kappa\}.
 \label{eq:basic-classification}
\end{equation}

For each \(i<\kappa\), the natural correspondence from
\(AM(P_i)_\leq\) onto
\[
 (\mathcal C_i,\leq_{AM})
\]
given by (\ref{eq:basic-Ci})--(\ref{eq:basic-Clast}) is an order
isomorphism.

Moreover,
\begin{equation}
 \mathcal C_{i-1}<B_i<M_i\leq_{AM}\mathcal C_i
 \qquad(1\leq i<\kappa),
 \label{eq:basic-order}
\end{equation}
where the first and last inequalities mean that every member of the
family on the left precedes every member of the family on the right.

Indeed, every vertex of \(V_{i-1}\) is below every vertex of
\(M_i\setminus\{x_i\}\), while \(y_i\) occurs in the members of
\(\mathcal C_{i-1}\) other than \(M_{i-1}\); and \(B_i<M_i\) because
\(y_i<x_i\) and the remaining elements of \(B_i\) already belong to
\(M_i\).

The set \(M_i\) is the least maximal antichain of \(P_i\). Indeed, if
\(A\in AM(P_i)\) and \(m\in M_i\), maximality of \(A\) gives an
\(a\in A\) comparable with \(m\), and minimality of \(m\) implies
\[
 m\leq a.
\]

It follows from (\ref{eq:basic-classification}) and
(\ref{eq:basic-order}) that
\begin{equation}
 AM(Q)_\leq
 \cong
 C_0\oplus
 \bigoplus_{1\leq i<\kappa}
 (\mathbf1\oplus C_i).
 \label{eq:basic-AM-raw}
\end{equation}

We now simplify (\ref{eq:basic-AM-raw}). 
%

For every \(1\leq i<\kappa\),
\begin{equation}
 \mathbf1\oplus C_i\cong C_i
 \quad\Longleftrightarrow\quad
 C_i\text{ is infinite}.
 \label{eq:basic-absorption}
\end{equation}
If \(C_i\) is finite, then \(\mathbf1\oplus C_i\) has one more element
than \(C_i\), so the two chains are not isomorphic.

By the definition of \(\widehat C_i\), equation
(\ref{eq:basic-AM-raw}) therefore gives
\[
 AM(Q)_\leq
 \cong
 C_0\oplus
 \bigoplus_{1\leq i<\kappa}\widehat C_i,
\]
which proves (\ref{eq:basic-AM}).

If \(C_i\) is infinite for every \(1\leq i<\kappa\), then
\[
 \widehat C_i=C_i
 \qquad(1\leq i<\kappa),
\]
and hence
\[
 AM(Q)_\leq
 \cong
 C_0\oplus
 \bigoplus_{1\leq i<\kappa}C_i
 =
 \sum_{i<\kappa}C_i
 =
 \sum_{i<\kappa}AM(P_i)_\leq.
\]
This proves (\ref{eq:basic-AM-infinite}).

\medskip
\noindent
\textbf{Claim 3: \(Q\) is prime.}

Let
\[
 W:=\bigcup_{i<\kappa}V_i,
\]
and suppose that \(N\) is a nontrivial module of \(Q\).

We first show that
\begin{equation}
 |N\cap W|\geq2.
 \label{eq:basic-module-two}
\end{equation}

If \(N\cap W=\varnothing\), then \(N\) contains two distinct vertices
\(y_i<y_j\). If \(j=i+1\), choose a nonminimal vertex \(v\in V_i\)
(such a vertex exists because a prime poset with at least four
vertices is not an antichain). Then
\[
 y_i<v
 \qquad\text{and}\qquad
 v\parallel y_j.
\]
If \(j\geq i+2\), any vertex of
\(V_{j-1}\setminus\{x_{j-1}\}\) is above \(y_i\) and incomparable
with \(y_j\). Thus an element outside \(N\) distinguishes two elements
of \(N\), a contradiction.

Suppose next that \(N\cap W=\{z\}\). We first show that \(N\) contains
at most one \(y_j\). If \(y_i,y_j\in N\), with \(i<j\) and
\(j=i+1\), then both a nonminimal vertex of \(V_i\) and any vertex of
\(M_j\setminus\{x_j\}\) distinguish \(y_i\) from \(y_j\). These
vertices lie in different blocks, so at least one of them is different
from \(z\) and hence lies outside \(N\), a contradiction.

If \(j\geq i+2\), then every vertex of
\(V_{j-1}\setminus\{x_{j-1}\}\) distinguishes \(y_i\) from \(y_j\).
Since \(|V_{j-1}|\geq4\), at least one such vertex is different from
\(z\), and again lies outside \(N\), a contradiction.

Therefore \(N\) contains exactly one new vertex, say \(y_j\), and
hence
\[
 N=\{z,y_j\}.
\]
If \(z\notin V_j\), choose
\[
 w\in M_j\setminus\{x_j\}.
\]
Then \(w\) is comparable with \(z\), because \(z\) belongs to a
different block, whereas
\[
 w\parallel y_j.
\]
If \(z\in V_j\), choose
\[
 w\in V_{j-1}\setminus\{x_{j-1}\}.
\]
Then
\[
 w<z,
 \qquad
 w\parallel y_j.
\]
Again \(w\) distinguishes the two elements of \(N\), a contradiction.
This proves (\ref{eq:basic-module-two}).

Now \(N\cap W\) is a module of the linear sum
\[
 \sum_{i<\kappa}P_i.
\]
If \(N\cap W\neq W\), then
Lemma~\ref{lem:module-linearsum} yields
\begin{equation}
 N\cap W=\bigcup_{j\in J}V_j
 \label{eq:basic-module-blocks}
\end{equation}
for a proper interval \(J\) of \(\kappa\).

Suppose first that \(J\) has a largest element \(\nu\) and
\(\nu+1<\kappa\). Then \(y_{\nu+1}\in N\): otherwise
\(y_{\nu+1}\) distinguishes \(x_\nu\) from the other minimal elements
of \(V_\nu\subseteq N\).

Choose
\[
 z\in M_{\nu+1}\setminus\{x_{\nu+1}\}.
\]
Since \(\nu+1\notin J\), we have \(z\notin N\). The vertex \(z\) is
above every element of \(V_\nu\subseteq N\), while
\[
 z\parallel y_{\nu+1}.
\]
Hence \(z\) distinguishes two elements of \(N\), a contradiction.

It remains to consider the case in which there is no right boundary
of \(J\) inside \(\kappa\). Since \(J\) is a proper interval of a
finite ordinal or of \(\omega\), \(J\) then has a least element
\(\varepsilon>0\).

The element \(y_\varepsilon\) belongs to \(N\): otherwise
\(y_\varepsilon\) distinguishes \(x_\varepsilon\) from any element of
\[
 M_\varepsilon\setminus\{x_\varepsilon\},
\]
both of which belong to \(V_\varepsilon\subseteq N\).

Choose
\[
 z\in V_{\varepsilon-1}\setminus\{x_{\varepsilon-1}\}.
\]
Since \(\varepsilon-1\notin J\), we have \(z\notin N\). Moreover,
\(z\) is below every element of \(V_\varepsilon\subseteq N\), whereas
\[
 z\parallel y_\varepsilon.
\]
Thus \(z\) distinguishes two elements of \(N\), again a contradiction.

Therefore
\[
 N\cap W=W.
\]

Finally, if some \(y_i\notin N\), then \(y_i\) distinguishes
\(x_{i-1}\) from any element of
\[
 M_{i-1}\setminus\{x_{i-1}\},
\]
both of which belong to \(V_{i-1}\subseteq N\). Hence every \(y_i\)
belongs to \(N\), and therefore
\[
 N=V(Q),
\]
contradicting that \(N\) was nontrivial. Thus \(Q\) is prime.

\medskip
\noindent
\textbf{Claim 4: \(Q\) is an interval order.}

We use Lemma~\ref{lem:intervalorder-prop}(iv). For
\(1\leq i<\kappa\), the relations in the construction give
\begin{equation}
 U_Q(y_i)=U_Q(x_i)\cup\{x_i\}.
 \label{eq:basic-Uy}
\end{equation}

Within each block \(V_i\), the sets
\[
 U_Q(v),\qquad v\in V_i,
\]
are linearly ordered by inclusion. Indeed, the sets \(U_{P_i}(v)\)
are linearly ordered, and the only new vertex which lies above one
vertex of \(V_i\) but not all the others is \(y_{i+1}\), when it
exists, which lies above \(x_i\) only. Since
\[
 U_{P_i}(x_i)=V_i\setminus M_i,
\]
the set \(U_Q(x_i)\) is maximal among the upper sets coming from
\(V_i\).

Moreover, for every \(i\geq1\),
\begin{equation}
 U_Q(v)\subseteq U_Q(y_i)\subseteq U_Q(u)
 \qquad
 (v\in V_i,\ u\in V_{i-1}).
 \label{eq:basic-upper-between}
\end{equation}
The first inclusion follows from (\ref{eq:basic-Ux}) and
(\ref{eq:basic-Uy}): every element above a vertex of \(V_i\) is above
\(y_i\). For the second, every element of \(U_Q(y_i)\) belongs to
\(V_i\), to a later block, or is a later \(y_j\), and hence is above
every vertex of \(V_{i-1}\).

Equations (\ref{eq:basic-Uy}) and
(\ref{eq:basic-upper-between}), together with the linear ordering of
the upper sets inside each \(V_i\), show that
\[
 \bigl(\{U_Q(z):z\in V(Q)\},\subseteq\bigr)
\]
is a chain. By Lemma~\ref{lem:intervalorder-prop}(iv), \(Q\) is an
interval order.

Combining the four claims proves the lemma.
\end{proof}

With this lemma, we can now prove Theorem~\ref{thm:main-rank}.

\begin{proof}
We proceed by transfinite induction on the Hausdorff rank. We shall prove
the slightly stronger statement that, for every countable ordinal
\(\alpha\), there exists a well-quasi-ordered prime interval order
\(P_\alpha\) such that
\begin{equation}
AM(P_\alpha)_\leq\cong\omega^{1+\alpha}.
\label{eq:canonical-rank-chain}
\end{equation}

We first verify that the chain on the right-hand side of
(\ref{eq:canonical-rank-chain}) has the required Hausdorff rank.

\medskip
\noindent
\textbf{Claim.}
For every ordinal \(\alpha\),

$$
 \operatorname{hr}\bigl(\omega^{1+\alpha}\bigr)=\alpha,
$$

where \(\operatorname{hr}\) denotes the Hausdorff rank used in this
paper.

\smallskip
\noindent
Put

$$
 H_\alpha:=\omega^{1+\alpha}
$$

and, for \(\gamma\leq\alpha\), put

$$
 \theta_\gamma:=\omega^{1+\gamma}.
$$

We use the description of Hausdorff rank by neighborhoods from
\cite[Sections~6.2.2--6.2.4]{fraissetr}. Recall that two points are
\(0\)-neighbors when the interval between them is finite; two points are
\((\gamma+1)\)-neighbors when there are only finitely many
\(\gamma\)-neighborhoods between them; and, if \(\lambda\) is a limit
ordinal, two points are \(\lambda\)-neighbors when they are
\(\gamma\)-neighbors for some \(\gamma<\lambda\).

We prove, by induction on \(\gamma\leq\alpha\), that the
\(\gamma\)-neighborhoods of \(H_\alpha\) are precisely the nonempty
intervals
\begin{equation}
I_{\gamma,\xi}
:=
[\,\theta_\gamma\xi,\,
   \theta_\gamma(\xi+1)\,)
\cap H_\alpha.
\label{eq:gamma-block}
\end{equation}

For \(\gamma=0\), we have \(\theta_0=\omega\). Thus the intervals

$$
 [\,\omega\xi,\omega(\xi+1)\,)
 \cap H_\alpha
$$

are exactly the maximal intervals in which every two points have a
finite interval between them. Hence they are precisely the
\(0\)-neighborhoods.

Suppose now that the assertion holds for \(\gamma\), where
\(\gamma+1\leq\alpha\). By the induction hypothesis, the
\(\gamma\)-neighborhoods are the consecutive blocks
\(I_{\gamma,\xi}\). Two points are \((\gamma+1)\)-neighbors precisely
when the indices of the \(\gamma\)-neighborhoods containing them have
only finitely many indices between them. Hence the equivalence classes
of \(\gamma\)-neighborhoods which form a
\((\gamma+1)\)-neighborhood are

$$
 I_{\gamma,\omega\eta},
 I_{\gamma,\omega\eta+1},
 I_{\gamma,\omega\eta+2},
 \ldots
$$

for some ordinal \(\eta\). Their union is

$$
 [\,\theta_\gamma\omega\eta,\,
    \theta_\gamma\omega(\eta+1)\,).
$$

Since

$$
 \theta_\gamma\omega
 =
 \omega^{1+\gamma}\omega
 =
 \omega^{(1+\gamma)+1}
 =
 \omega^{1+(\gamma+1)}
 =
 \theta_{\gamma+1},
$$

this union is exactly

$$
 [\,\theta_{\gamma+1}\eta,\,
    \theta_{\gamma+1}(\eta+1)\,).
$$

Thus the \((\gamma+1)\)-neighborhoods are precisely the intervals
\(I_{\gamma+1,\eta}\).

It remains to consider a nonzero limit ordinal
\(\lambda\leq\alpha\). Assume that
(\ref{eq:gamma-block}) holds for every \(\gamma<\lambda\).
Since ordinal addition is continuous in its second argument,

$$
 1+\lambda=\sup_{\gamma<\lambda}(1+\gamma),
$$

and therefore, by continuity of ordinal exponentiation,
\begin{equation}
\theta_\lambda
=
\omega^{1+\lambda}
=
\sup_{\gamma<\lambda}\omega^{1+\gamma}
=
\sup_{\gamma<\lambda}\theta_\gamma.
\label{eq:theta-limit}
\end{equation}

We first show that two points in the same
\(\theta_\lambda\)-block are \(\lambda\)-neighbors. Let

$$
 x<y
$$

belong to the same interval \(I_{\lambda,\xi}\). Then there are
ordinals \(\rho<\sigma<\theta_\lambda\) such that

$$
 x=\theta_\lambda\xi+\rho,
 \qquad
 y=\theta_\lambda\xi+\sigma.
$$

By (\ref{eq:theta-limit}), there exists \(\gamma<\lambda\) such that

$$
 \sigma<\theta_\gamma.
$$

Since

$$
 1+\gamma<1+\lambda,
$$

there exists an ordinal \(\delta>0\) such that

$$
 1+\lambda=(1+\gamma)+\delta.
$$

Consequently,

$$
 \theta_\lambda
 =
 \omega^{1+\lambda}
 =
 \omega^{(1+\gamma)+\delta}
 =
 \omega^{1+\gamma}\omega^\delta
 =
 \theta_\gamma\omega^\delta.
$$

By associativity of ordinal multiplication,

$$
 \theta_\lambda\xi
 =
 \theta_\gamma(\omega^\delta\xi).
$$

Thus \(\theta_\lambda\xi\) is a boundary between
\(\theta_\gamma\)-blocks. Since

$$
 \rho<\sigma<\theta_\gamma,
$$

both \(x\) and \(y\) belong to the same interval

$$
 I_{\gamma,\omega^\delta\xi}.
$$

By the induction hypothesis, \(x\) and \(y\) are
\(\gamma\)-neighbors. Since \(\gamma<\lambda\), the definition at a
limit stage implies that \(x\) and \(y\) are
\(\lambda\)-neighbors.

Conversely, suppose that \(x\) and \(y\) lie in distinct
\(\theta_\lambda\)-blocks. Let \(\gamma<\lambda\). As above, choose
\(\delta>0\) such that

$$
 1+\lambda=(1+\gamma)+\delta.
$$

Then

$$
 \theta_\lambda=\theta_\gamma\omega^\delta,
$$

and hence every boundary

$$
 \theta_\lambda\xi
$$

between two \(\theta_\lambda\)-blocks is also a boundary between
\(\theta_\gamma\)-blocks, since

$$
 \theta_\lambda\xi
 =
 \theta_\gamma(\omega^\delta\xi).
$$

Therefore a \(\theta_\gamma\)-block cannot contain points lying in two
distinct \(\theta_\lambda\)-blocks. By the induction hypothesis,
\(x\) and \(y\) are not \(\gamma\)-neighbors. Since this holds for
every \(\gamma<\lambda\), the definition of a limit neighborhood
implies that \(x\) and \(y\) are not \(\lambda\)-neighbors.

Hence the \(\lambda\)-neighborhoods are precisely the intervals
\(I_{\lambda,\xi}\). This completes the induction proving
(\ref{eq:gamma-block}).

Taking \(\gamma=\alpha\), we have

$$
 \theta_\alpha=\omega^{1+\alpha}=H_\alpha,
$$

so \(H_\alpha\) consists of a single \(\alpha\)-neighborhood. On the
other hand, if \(\gamma<\alpha\), then

$$
 \theta_\gamma<\theta_\alpha,
$$

and hence \(H_\alpha\) contains more than one
\(\gamma\)-neighborhood. Therefore

$$
 \operatorname{hr}(H_\alpha)=\alpha.
$$

This proves the claim.

\medskip
We now construct \(P_\alpha\).

For \(\alpha=0\), the ordered set \(I_{\NN}\) described above is a
well-quasi-ordered prime interval order and

$$
 AM(I_{\NN})_\leq\cong\omega
 =\omega^{1+0}.
$$

Thus we may take

$$
 P_0:=I_{\NN}.
$$

Now let \(\alpha>0\), and suppose inductively that, for every
\(\beta<\alpha\), there exists a well-quasi-ordered prime interval
order \(P_\beta\) such that
\begin{equation}
AM(P_\beta)_\leq\cong\omega^{1+\beta}.
\label{eq:rank-induction-hypothesis}
\end{equation}
All the ordered sets constructed in the induction are infinite, and
hence have at least four elements. Thus
Lemma~\ref{lem:basi-construction} is applicable.

\medskip
\noindent
\textbf{Case 1: \(\alpha\) is a successor ordinal.}

Write

$$
 \alpha=\beta+1.
$$

For every \(n<\omega\), let \(P_n\) be a copy of \(P_\beta\), with the
copies chosen pairwise disjoint. Apply
Lemma~\ref{lem:basi-construction} with \(\kappa=\omega\) to the
sequence \((P_n)_{n<\omega}\). The resulting ordered set \(Q\) is a
well-quasi-ordered prime interval order, and

$$
 AM(Q)_\leq
 \cong
 \sum_{n<\omega}AM(P_n)_\leq.
$$

By the induction hypothesis,

$$
 AM(Q)_\leq
 \cong
 \sum_{n<\omega}\omega^{1+\beta}
 =
 \omega^{1+\beta}\omega.
$$

Therefore

$$
 AM(Q)_\leq
 \cong
 \omega^{(1+\beta)+1}
 =
 \omega^{1+(\beta+1)}
 =
 \omega^{1+\alpha}.
$$

Hence we may take

$$
 P_\alpha:=Q.
$$

\medskip
\noindent
\textbf{Case 2: \(\alpha\) is a limit ordinal.}

Since \(\alpha\) is a nonzero countable limit ordinal, its cofinality is
\(\omega\). Choose a strictly increasing sequence

$$
 \beta_0<\beta_1<\beta_2<\cdots<\alpha
$$

which is cofinal in \(\alpha\). Thus

$$
 \sup_{n<\omega}\beta_n=\alpha.
$$

For every \(n<\omega\), let \(P_n\) be a copy of
\(P_{\beta_n}\), with the copies chosen pairwise disjoint. Apply
Lemma~\ref{lem:basi-construction} with \(\kappa=\omega\). The resulting
ordered set \(Q\) is a well-quasi-ordered prime interval order, and

$$
 AM(Q)_\leq
 \cong
 \sum_{n<\omega}AM(P_n)_\leq.
$$

By the induction hypothesis,
\begin{equation}
AM(Q)_\leq
\cong
\sum_{n<\omega}\omega^{1+\beta_n}.
\label{eq:limit-rank-sum}
\end{equation}

The sequence \((1+\beta_n)_{n<\omega}\) is strictly increasing.
Therefore, for every \(n<\omega\),

$$
 \omega^{1+\beta_0}
 +\omega^{1+\beta_1}
 +\cdots
 +\omega^{1+\beta_n}
 =
 \omega^{1+\beta_n}.
$$

Consequently,

$$
 \sum_{n<\omega}\omega^{1+\beta_n}
 =
 \sup_{n<\omega}\omega^{1+\beta_n}.
$$

By continuity of ordinal exponentiation,

$$
 \sup_{n<\omega}\omega^{1+\beta_n}
 =
 \omega^{\,\sup_{n<\omega}(1+\beta_n)}.
$$

Since ordinal addition is continuous in its second argument and
\((\beta_n)_{n<\omega}\) is cofinal in \(\alpha\),

$$
 \sup_{n<\omega}(1+\beta_n)
 =
 1+\sup_{n<\omega}\beta_n
 =
 1+\alpha.
$$

Hence

$$
 AM(Q)_\leq
 \cong
 \omega^{1+\alpha}.
$$

We may therefore take

$$
 P_\alpha:=Q.
$$

In both cases we have constructed a well-quasi-ordered prime interval
order \(P_\alpha\) such that

$$
 AM(P_\alpha)_\leq\cong\omega^{1+\alpha}.
$$

By the claim, the chain \(\omega^{1+\alpha}\) has Hausdorff rank
\(\alpha\). Hence \(AM(P_\alpha)_\leq\) has Hausdorff rank
\(\alpha\), as required.
\end{proof}

\section*{Acknowledgements}

The authors would like to thank the anonymous referees for their helpful
comments and suggestions, which improved the presentation of the paper. In
particular, we would like to thank one referee for suggesting an alternative
proof of part of Theorem~\ref{thm:prime-scattered} and for helping simplify
parts of the proof of Lemma~\ref{lem:basi-construction}.

%
%
%

%


\end{document}